\documentclass[a4paper,12]{article}
\usepackage{amsmath, amsfonts, amssymb, amsthm, bm, graphics, bbm, color}
\usepackage{graphicx}
\usepackage{subfigure}
\usepackage{mathtools}
\usepackage[table]{xcolor}
\usepackage{url}
\usepackage{mathabx}
\usepackage{cancel}
\usepackage{multicol}
\usepackage{float}
\usepackage{caption}
\usepackage{mathalfa}
\usepackage{hyperref}
\usepackage{afterpage}
\usepackage{multirow}
\usepackage[section]{placeins}
\newcommand*{\myprime}{^{\prime}\mkern-1.2mu}
\newcommand*{\mydprime}{^{\prime\prime}\mkern-1.2mu}

\DeclareMathAlphabet\mathbfcal{OMS}{cmsy}{b}{n}
\newtheorem{theorem}{Theorem}[section]
\newtheorem{lemma}[theorem]{Lemma}

\theoremstyle{definition}

\newtheorem{remark}[theorem]{Remark}

\newcolumntype{P}[1]{>{\centering\arraybackslash}p{#1}}

\makeatletter
\newcommand\makebig[2]{%
  \@xp\newcommand\@xp*\csname#1\endcsname{\bBigg@{#2}}%
  \@xp\newcommand\@xp*\csname#1l\endcsname{\@xp\mathopen\csname#1\endcsname}%
  \@xp\newcommand\@xp*\csname#1r\endcsname{\@xp\mathclose\csname#1\endcsname}%
}
\makeatother

\makebig{biggg} {3.0}
\makebig{Biggg} {3.5}
\makebig{bigggg}{4.0}
\makebig{Bigggg}{4.5}
\makebig{biggggg}{5.0}
\makebig{Biggggg}{5.5}
\makebig{bigggggg}{6.0}
\makebig{Bigggggg}{6.5}

\renewcommand{\vec}[1]{\mbox{\boldmath$#1$}}

\newcommand{\dif}{\mathrm{d}}
\newcommand{\im}{\mathrm{i}}

\begin{document}
\title{{Characterising buried objects in metal detection}}
\author{P.D. Ledger$^*$ and W.R.B. Lionheart$^\dagger$,\\
$^*$School of Computing \& Mathematical Sciences, University of Leicester,\\
University Road, Leicester LE1 7RH, U.K.\\
$^\dagger$Department of Mathematics, The University of Manchester,\\
Oxford Road, Manchester, M13 9PL, U.K. \\
Corresponding author: pdl11@leicester.ac.uk}
\maketitle

\section*{Abstract}
Current mathematical models for identifying highly conducting buried objects in metal detection  assume that the soil is non-conducting and {has the same permeability  as  free space. However, although the electrical conductivity of soil is low, it is not negligible and depends on factors such as soil type and salinity.
Moreover, the magnetic permeability of soil varies with its iron content and is often described as uncooperative.}
Depending on the ground conditions,  {these soil's properties} can influence the induced voltages in the measurement coils of metal detectors and becomes of increasing importance as the frequency of the exciting current source  is increased. In this work, we develop a new asymptotic expansion for the perturbed magnetic field due to the presence of a highly conducting magnetic buried object as its size tends to zero, which takes account of the ground conditions.  The leading order term of this expansion can be expressed in terms of a complex symmetric rank-2 magnetic polarizability tensor, {which characterises the object, can assist in its identification}, and we provide conditions under which this tensor's coefficients can be computed independently of the ground conditions. We {demonstrate} the improved accuracy of our new result, which takes account of the ground conditions, over  the situation where the soil's conductivity and permeability are not considered.

\noindent {\bf Keywords:} Metal detection; Polarizability tensors; Eddy current; Buried objects; Uncooperative soil.

\noindent {\bf 2020 Mathematics Subject Classification:} 35R30, 35Q61, 78A46 (Primary); 35B30, 78A25 (Secondary)

\section{Introduction}
Metal detectors provide a low-cost portable solution for identifying buried metallic objects located close to the ground's surface. Applications include identifying unexploded ordnance, anti-personnel {and anti-vehicle}  landmines in {former conflict zones}, mineral prospecting and finding metallic items of archeological {or forensic} significance. They  are also popular with hobbyist treasure hunters. Metal detectors work on the principle of magnetic induction and often consist a simple audible tone obtained from the induced voltage, which an operator uses to identify the presence of a conducting body. {Key challenges} for existing metal detectors include  identifying the shape of a conducting body and distinguishing between a small object buried close to the surface and a larger object buried at greater depths. However, since the voltage signal induced in a measurement coil contains object characterisation information,  there is scope to use this information to help to better identify hidden objects.

Metallic objects have a high electrical conductivity (typically of the order of $10^6$ S/m or higher) and magnetic metals have a relative magnetic permeability that is substantially greater than $1$. Given the high conductivities of metals and the low frequencies of metal detectors (typically $10-100$ kHz), mathematical models of metal detection  use the eddy current approximation of Maxwell's equations, where the displacement currents are neglected due to the dominance of the Ohmic currents~\cite{Ammari2014}. It is also common to assume that the ground is non-conducting and have to the same magnetic permeability as free space. {In this work,  this important restriction is removed and new results on object characterisation are obtained. }

While the conductivity of the soil is much smaller than that of the metallic object, its electrical conductivity is not zero and depends on the specific nature of the soil. Strongly saline soils, {which includes dissolved salts that are naturally occurring as well as contamination from pollution or fertilisers},   have an electrical conductivity of around 1.6 S/m~\cite{salinesoilcond}, but those soils with a low salt content have a lower electrical conductivity. Clay  soils tend to have a higher electrical conductivity than those primarily made of silt or sand. While the relative magnetic permeability of many soils is typically close to 1 (typical value 1.0006), uncooperative soils and rocks with high iron content can have higher values e.g. volcanic rock and soil (typical value 1.021), granite (typical value 1.076) and rock {in} some iron mining areas (typical value 1.1)~\cite{magsoilperm}. Depending on the ground conditions, the electrical conductivity and magnetic permeability of the soil can influence the signal produced and are of increasing importance as the frequency is increased.

Using the eddy current approximation of Maxwell's equations, and assuming the conducting object is placed in a non-conducting background,  Ammari, Chen, Chen, Garnier and Volkov~\cite{Ammari2014} have developed the leading order term in an asymptotic expansion of the perturbed magnetic field due to the presence of a highly conducting object as its size tends to zero. Ledger and Lionheart~\cite{LedgerLionheart2015} have shown that this leading order term can be written in terms of a complex symmetric rank-2 magnetic polarizability tensor (MPT), which is independent of the object position and provide explicit expressions for the computation of its coefficients. 
The authors have also obtained a complete asymptotic expansion of the perturbed magnetic field where the higher order terms are written in terms of generalised MPTs~\cite{LedgerLionheart2018g,LedgerLionheart2023a}, which provide additional object characterisation information. The (generalised) MPTs are related to the generalised polarisation/polarizability tensors introduced by Ammari and Kang~\cite{ammarikangbook} for applications in electrical impedance tomography and composite materials, but require the solution of vectorial, rather than scalar, transmission problems and their coefficients depend on the object's size, shape, electrical conductivity and magnetic permeability as well as the exciting frequency.
The MPT coefficients as a function of exciting frequency, known as its spectral signature~\cite{LedgerLionheart2020spect},  also provide additional characterisation information and computational approaches~\cite{ben2020,Elgy2024_preprint} have been developed to efficiently compute the MPT spectral signature. Dictionaries of computed MPTs have been obtained~\cite{ledgerwilsonamadlion2021} and machine learning approaches applied to identify classes of objects for applications in security screening~\cite{marsh2014b, marsh2015,ledgerwilsonlion2022}.

The novelties of this work are as follows:

\begin{itemize}

\item The development of a new asymptotic expansion for the perturbed magnetic field evaluated in free space due to the presence of highly conducting (magnetic) buried object as its size tends to zero and the soil is a low-conducting medium.

\item Providing conditions under which the object can be characterised by a rank-2 MPT whose coefficients can be computed independently of the ground conditions.

\item Illustrating the improved accuracy obtained by the new asymptotic expansion, which incorporates information about the ground conditions, compared to a previous formulation that assumes the properties of the ground are the same as free space.

\end{itemize}

Fundamental to our work, will be the use of an appropriate Green's function to take account of the soil-air interface. In general this will necessitate the use of a dyadic Green's function rather than the scalar Laplace Green's function considered in previous work~\cite{Ammari2014,LedgerLionheart2015}.  To aid with this, we will draw on the work in the area of non-destructive testing where Green's functions have been obtained for half-space problems involving a conductor-air interface~\cite{bowler1987,bowler1991}.

The material will be presented as follows: In Section~\ref{sect:mathmodel}, the mathematical model for a buried highly conducting permeable object using the eddy current approximation of the Maxwell system is described. Then, in Section~\ref{sect:energy}, new energy estimates are obtained extending those presented in~\cite{Ammari2014} to the situation considered in this work.  Section~\ref{sect:transprob} presents transmission problems that are independent of the object's position. Then, in Section~\ref{sect:integalrep}, integral representation formulae are derived where the aforementioned dyadic Green's function that  takes account of the soil-air interface plays an important role.  Section~\ref{sect:asymform} derives the new asymptotic formula for the perturbed magnetic field as the size of the buried object tends to zero and Section~\ref{sect:tensor} presents this in terms of a magnetic polarizability tensor representation that can be computed independently of the soil's properties and the object's position. In Section~\ref{sect:numerical} we include a series of numerical results to illustrate our proposed approach.

\section{Mathematical model}\label{sect:mathmodel} 

\subsection{Governing equations}
Given a highly conducting permeable object $B_\alpha$  with Lipschitz boundary $\Gamma_*:=\partial B_\alpha$ buried in a low-conducting medium $\Omega_{s}\subset {\mathbb R}^3 $, with $B_\alpha  \subset \Omega_{s} $, and excited by a solenoidal time harmonic solenoidal current source ${\vec J}^s$ located in free space $\Omega_{fs} := {\mathbb R}^3 \setminus \overline{\Omega_{s}}$, we wish to predict the magnetic field perturbation in $\Omega_{fs}$ due to the presence of $B_\alpha$. We describe the object as $B_\alpha= \alpha B + {\vec z}$ where $\alpha$ denotes the object size, $B$ describes a uniform sized object placed at origin and ${\vec z}$ denotes a translation. We will further assume that the ground's surface $\Gamma_s:= \partial \Omega_s \cap \partial \Omega_{fs} $  is horizontal. 

In the absence of the object, the background electric and magnetic fields ${\vec E}_0$ and ${\vec H}_0$, respectively, satisfy
\begin{subequations}
\begin{align}
\nabla \times  {\vec E}_0  &= \im \omega \mu_0 {\vec H}_0 && \text{in $\Omega_{fs}$}, \\
\nabla \times  {\vec H}_0  &=  {\vec J}^s && \text{in $\Omega_{fs}$}, \\
\nabla \times  {\vec E}_0  &= \im \omega \mu_{r,s} \mu_0 {\vec H}_0 && \text{in $\Omega_{s}$}, \\
\nabla \times  {\vec H}_0  &= \sigma_{s} {\vec E}_0  && \text{in $\Omega_{s}$}, \\
\nabla \cdot {\vec E}_0 = \nabla \cdot {\vec H}_0 & = 0 && \text{in ${\mathbb R}^3 $}, \\
[ {\vec n} \times {\vec E}_0 ]_{\Gamma_s} =  [{\vec n} \times {\vec H}_0 ]_{\Gamma_s} & = {\vec 0}  && \text{on $\Gamma_{s}$}, \\
{\vec E}_0  = O(1/|{\vec x}|), {\vec H}_0  &= O(1/|{\vec x}|),&& \text{as $|{\vec x}| \to \infty$}, 
\end{align}
\end{subequations}
where $[\cdot]$ denotes the jump, ${\vec n}$ is the unit outward normal, $\omega$ is the angular frequency, $\im = \sqrt{-1}$, $\mu_0 = 4\pi \times 10^{-7}$H/m is the permeability of free space, $\mu_{r,s}$ is the relative magnetic permeability of the soil and $\sigma_s$ is the electrical conductivity of the soil.   {We shall assume that the material parameters of the soil are constant. The uniqueness of ${\vec E}_0$ in $\Omega_{fs}$ is achieved by additionally specifying 
\begin{align}
\int_{\Gamma_s} {\vec n} \cdot {\vec E}_0 |_+ \dif {\vec x}, \label{eqn:unique}
\end{align}
and in practice the decay of the fields is faster than $|{\vec x}|^{-1}$~\cite{ammaribuffa2000}. } Eliminating ${\vec H}_0$ gives
\begin{subequations} \label{eqn:e0}
\begin{align}
\nabla \times \mu_0^{-1}  \nabla \times {\vec E}_0 &=  \im \omega {\vec J}^s && \text{in $\Omega_{fs}$} ,\\
\nabla \times \mu_0^{-1} \mu_{r,s}^{-1}  \nabla\times {\vec E}_0 - \im \omega  \sigma_s {\vec E}_0 &=  {\vec 0} && \text{in $\Omega_{s}$},\\
\nabla \cdot {\vec E}_0  & = 0 && \text{in $\Omega_{fs}\cup \Omega_s $},\\
[ {\vec n} \times {\vec E}_0 ]_{\Gamma_s}  = [ {\vec n} \times \mu^{-1} \nabla \times  {\vec E}_0 ]_{\Gamma_s}& = {\vec 0} &&  \text{on $\Gamma_{s} $} ,\\
{\vec E}_0  &= O(1/|{\vec x}|) && \text{as $|{\vec x}| \to \infty$}.
\end{align}
\end{subequations}
Note the divergence constraint is automatically satisfied in $\Omega_s$, but we retain it below for convenience of considering the limiting case where the soil is non-conducting. 

In the presence of the object, the electric and magnetic interaction fields  ${\vec E}_\alpha$ and ${\vec H}_\alpha$ satisfy
\begin{subequations}
\begin{align}
\nabla \times  {\vec E}_\alpha  &= \im \omega \mu_0 {\vec H}_\alpha && \text{in $\Omega_{fs}$}, \\
\nabla \times  {\vec H}_\alpha  &=  {\vec J}^s && \text{in $\Omega_{fs}$}, \\
\nabla \times  {\vec E}_\alpha  &= \im \omega \mu_0 \mu_{r,s}  {\vec H}_\alpha && \text{in $\Omega_{s}\setminus \overline{B}_\alpha $}, \\
\nabla \times  {\vec H}_\alpha  &= \sigma_{s} {\vec E}_\alpha  && \text{in $\Omega_{s} \setminus \overline{B_\alpha}$}, \\
\nabla \times  {\vec E}_\alpha  &= \im \omega \mu_0 \mu_{r,*}  {\vec H}_\alpha && \text{in ${B}_\alpha $}, \\
\nabla \times  {\vec H}_\alpha  &= \sigma_{*} {\vec E}_\alpha  && \text{in ${B_\alpha}$}, \\
\nabla \cdot {\vec E}_\alpha= \nabla \cdot {\vec H}_\alpha & = 0 && \text{in $\Omega_{fs} \cup \Omega_s$}, \\
[ {\vec n} \times {\vec E}_\alpha ]_{\Gamma} =  [{\vec n} \times {\vec H}_\alpha ]_{\Gamma} & = {\vec 0}  && \text{on $\Gamma=\Gamma_{s},\Gamma_* $}, \\
{\vec E}_\alpha  = O(1/|{\vec x}|), {\vec H}_\alpha  &= O(1/|{\vec x}|),&& \text{as $|{\vec x}| \to \infty$}, 
\end{align}
\end{subequations}
where $\mu_{r,*}$ is the relative magnetic permeability of the object and $\sigma_*$ is the electrical conductivity of the object, which are assumed to be constants. 
 {The uniqueness of ${\vec E}_\alpha$ in $\Omega_{fs}$ is achieved by additionally specifying a condition similar to (\ref{eqn:unique}) with ${\vec E}_0$ replaced by ${\vec E}_\alpha$  and again, in practice, the decay of the fields is faster than $|{\vec x}|^{-1}$~\cite{ammaribuffa2000}. If $\sigma_s=0$, the uniqueness of ${\vec E}_\alpha $ in $\Omega_{fs} \cup \Omega_s \setminus \overline{B_\alpha}$ is achieved by a similar integral, but instead integrated over $\Gamma_*$.}
Eliminating ${\vec H}_\alpha$ gives
\begin{subequations} \label{eqn:ealpha}
\begin{align}
\nabla \times \mu_0^{-1}  \nabla \times {\vec E}_\alpha &= \im \omega  {\vec J}^s && \text{in $\Omega_{fs}$} ,\\
\nabla \times \mu_0^{-1} \mu_{r,s}^{-1}  \nabla\times {\vec E}_\alpha - \im \omega  \sigma_s {\vec E}_\alpha &=  {\vec 0} && \text{in $\Omega_{s} \setminus \overline{B}_\alpha $},\\
\nabla \times \mu_0^{-1}  \mu_{r,*}^{-1}  \nabla\times {\vec E}_\alpha - \im \omega  \sigma_* {\vec E}_\alpha &=  {\vec 0} && \text{in $ {B}_\alpha $},\\
\nabla \cdot {\vec E}_\alpha   & = 0 && \text{in $\Omega_{fs} \cup \Omega_s $},\\
[ {\vec n} \times {\vec E}_\alpha ]_{\Gamma} =  [ {\vec n} \times \mu^{-1} \nabla \times  {\vec E}_\alpha  ]_{\Gamma} & = {\vec 0} &&  \text{on $\Gamma= \Gamma_{s}, \Gamma_* $} ,\\
{\vec E}_0  &= O(1/|{\vec x}|) && \text{as $|{\vec x}| \to \infty$}. 
\end{align}
\end{subequations}

Introducing
 {
\begin{align}
X :&= \{ {\vec u} \in {\vec H} ( \hbox{curl}) :  \nabla \cdot {\vec u} = 0 \text{ in $\Omega_{fs}\cup \Omega_s$}, {\vec u} = O(1/|{\vec x}|) \text{ as $|{\vec x}| \to \infty $} \}, \nonumber \\
\tilde{X} :& = \left \{ {\vec u} \in X :  \int_{\Gamma_s} {\vec n} \cdot {\vec u} |_+ \dif {\vec x} =0 \text{ if $\sigma_s \ne 0$ and }  \int_{\Gamma_*} {\vec n} \cdot {\vec u} |_+ \dif {\vec x} =0 
\text{ otherwise} \right \}, \nonumber
\end{align}
} the weak solution of (\ref{eqn:e0}) is: Find ${\vec E}_0 \in \tilde{X}$ such that
\begin{align}
( \mu_0^{-1} \nabla \times  {\vec E}_0 , \nabla \times {\vec v} )_{\Omega_{fs}} + ( \mu_0^{-1} \mu_{r,s}^{-1}  \nabla \times  {\vec E}_0 , \nabla \times {\vec v} )_{\Omega_{s}}
-\im \omega ( \sigma_s  {\vec E}_0 ,  {\vec v} )_{\Omega_{s}} = \im \omega  ( {\vec J}^s , {\vec v})_{\Omega_{fs}} \forall {\vec v} \in \tilde{X},  \label{eqn:we0}
\end{align}
where $(\cdot , \cdot)_{\Omega}$ is the standard $L^2$ inner product over $\Omega$.  {If $\sigma_s=0$, $\mu_r=1$ this reduces to the situation considered in~\cite{Ammari2014}}.
The weak solution of (\ref{eqn:ealpha}) is: Find ${\vec E}_\alpha  \in \tilde{X}$ such that
\begin{align}
( \mu_0^{-1} \nabla \times  {\vec E}_\alpha  , \nabla \times {\vec v} )_{\Omega_{fs}} +& ( \mu_0^{-1} \mu_{r,s}^{-1}  \nabla \times  {\vec E}_\alpha , \nabla \times {\vec v} )_{\Omega_{s}\setminus \overline{B_\alpha}}
+ ( \mu_0^{-1} \mu_{r,*}^{-1}  \nabla \times  {\vec E}_\alpha , \nabla \times {\vec v} )_{B_\alpha }\nonumber \\
&-\im \omega ( \sigma_s  {\vec E}_\alpha ,  {\vec v} )_{\Omega_{s} \setminus \overline{B_\alpha} }
-\im \omega ( \sigma_*  {\vec E}_\alpha ,  {\vec v} )_{B_\alpha}
 = \im \omega  ( {\vec J}^s , {\vec v})_{\Omega_{fs}} \forall {\vec v} \in \tilde{X},  \label{eqn:wealpha}
\end{align}
 {and again for $\sigma_s=0, \mu_{r,s}=1$ reduces to the situation considered in~\cite{Ammari2014} .}

In a similar manner to~\cite{Ammari2014}, we introduce 
\begin{align}
\nu :=   \alpha^2  \sigma_* \mu_0 \omega <C, 
\end{align}
and additionally require
{
\begin{align}
\varepsilon: =\omega \mu_0 \sigma_sD^2  \le \nu,
 \end{align}
where $D$ corresponds to the depth of the object, which controls $\sigma_s$ relative to the other parameters. In addition, we require  $1\le \mu_{r,s} \le 1+ \frac{\alpha}{D}$. This means that
$|1- \mu_{r,s}^{-1}| \omega \mu_0 \sigma_sD^2 \le \varepsilon$ and  $|\mu_{r,s}-1 | \le \frac{\alpha}{D} \le \left (\frac{\alpha}{D}\right )^2 \frac{\sigma_*}{\sigma_s} \mu_{r,s}$.}

\section{Energy estimates} \label{sect:energy}

Constructing (\ref{eqn:wealpha}) - (\ref{eqn:we0}) leads to
\begin{align}
(  \mu_0^{-1} \nabla \times ( {\vec E}_\alpha - {\vec E}_0) ,  \nabla \times {\vec v} )_{\Omega_{fs} }+&
(  \mu_0^{-1} \mu_{r,s}^{-1}  \nabla \times ( {\vec E}_\alpha - {\vec E}_0) , \nabla \times {\vec v} )_{\Omega_{s}\setminus \overline{B_\alpha} }
+
(  \mu_0^{-1} \mu_{r,*}^{-1} \nabla \times ( {\vec E}_\alpha - {\vec E}_0) ,  \nabla \times {\vec v})_{B_\alpha } \nonumber\\
&- \im \omega (  \sigma_s  ( {\vec E}_\alpha - {\vec E}_0) , {\vec v} )_{\Omega_s \setminus \overline{B_\alpha} }
- \im \omega (  \sigma_*  ( {\vec E}_\alpha - {\vec E}_0 ), {\vec v} )_{  {B_\alpha} } \nonumber \\
& = \mu_0^{-1} ( (\mu_{r,s}^{-1} - \mu_{r,*}^{-1}) \nabla \times   {\vec E}_0 ,  \nabla \times {\vec v} )_{B_\alpha}
-\im \omega ( (\sigma_s-\sigma_*)   {\vec E}_0 , {\vec v} )_{  {B_\alpha} } .\nonumber
\end{align}

Then, in a similar way to~\cite{Ammari2014}, we introduce ${\vec w}\in \tilde{X}$ as the solution to 
\begin{align}
(  \mu_0^{-1} \nabla \times {\vec w},  \nabla \times {\vec v} )_{\Omega_{fs} }+&
(  \mu_0^{-1} \mu_{r,s}^{-1} \nabla \times {\vec w}  , \nabla \times {\vec v} )_{\Omega_{s}\setminus \overline{B_\alpha} }
+
(  \mu_0^{-1} \mu_{r,*}^{-1}  \nabla \times {\vec w}  ,  \nabla \times {\vec v})_{B_\alpha } \nonumber\\
&- \im \omega (  \sigma_s  {\vec w}  , {\vec v} )_{\Omega_s \setminus \overline{B_\alpha} }
- \im \omega (  \sigma_*   {\vec w}  , {\vec v} )_{  {B_\alpha} } \nonumber \\
& = \mu_0^{-1} ( (\mu_{r,s}^{-1} - \mu_{r,*}^{-1}) \nabla \times {\vec F}  ,  \nabla \times {\vec v} )_{B_\alpha}
-\im \omega ( (\sigma_s-\sigma_*)  {\vec F} , {\vec v} )_{  {B_\alpha} } \forall {\vec v} \in \tilde{X}, \label{eqn:wprob}
\end{align}
where
\begin{subequations}\label{eqn:feqns}
\begin{align}
{\vec F}({\vec x}) & := \frac{1}{2} (\nabla_z \times {\vec E}_0 ) \times ({\vec x} - {\vec z}) + \frac{1}{3} {\vec D}_z (\nabla_z \times {\vec E}_0 )({\vec x} - {\vec z} ) \times ({\vec x} - {\vec z}),
\end{align}
so that
\begin{align}
\nabla_x \times {\vec F} & = \nabla_z \times {\vec E}_0 + {\vec D}_z (\nabla_z \times {\vec E}_0 )({\vec x} - {\vec z} ), \\
\nabla_x \cdot {\vec F} & = -\im \omega \mu_0 \nabla_z \times {\vec H}_0\cdot ({\vec x} - {\vec z})  = - \im \omega \mu_0 \sigma_s {\vec E}_0 ({\vec z})\cdot ({\vec x} - {\vec z}) , \label{eqn:divF}
\end{align}
\end{subequations}
in $B_\alpha$, where $\nabla_x \times {\vec F} $ corresponds to the first two terms of a Taylor's series expansion of $ \nabla \times {\vec E}_0$ about ${\vec z}$ as $| {\vec x} - {\vec z}| \to 0$. 

 {Considering first the case of $\sigma_s=0$, $\mu_{r,s}=1$ then 
 \begin{subequations}
\begin{align}
\nabla \cdot  {\vec w} & = - \nabla \cdot  {\vec F} && \text{in ${B}_\alpha$} ,\\
{\vec n} \cdot   {\vec w}^-  &  =  - {\vec n} \cdot   {\vec F}  &&  \text{on $\Gamma_*$} .
\end{align}
\end{subequations}
By choosing $\nabla \cdot( {\vec F} - \nabla \phi_0 ) =0$ in $B_\alpha$, setting
${\vec n} \cdot  ( \nabla \phi_0 +{\vec E}_0)^-  ={\vec n} \cdot  {\vec F}  ({\vec x})$ on $\Gamma_*$ and $\int_{B_\alpha} \phi_0 \dif {\vec y} =0 $, then $ \nabla \cdot   ( {\vec F} - \nabla \phi_0 ) =0$ and, hence, $\nabla \cdot ( - {\vec w} -  \nabla \phi_0 ) =0$ in $B_\alpha$ ~\cite{Ammari2014}. From (\ref{eqn:we0}) we have $\nabla \cdot {\vec E}_0 =0$ in $B_\alpha$ and ${\vec n} \cdot {\vec E}_0^- =0 $ on $\Gamma_*$, and, from (\ref{eqn:wealpha}),  $\nabla \cdot {\vec E}_\alpha  = 0 $ in $B_\alpha$ and ${\vec n} \cdot {\vec E}_\alpha^- =0 $ on $\Gamma_*$. This means that, 
\begin{align}
\| {\vec E}_0 - {\vec F} + \nabla \phi_0\|_{L^2(B_\alpha)} \le & C \alpha \left (  \| \nabla \times ({\vec E}_0 - {\vec F})\|_{L^2(B_\alpha)} + \| {\vec n} \cdot ({\vec E}_0 - {\vec F} + \nabla \phi_0 )\|_{L^2(\Gamma_* )}  \right ) \nonumber\\
\le & C \alpha   \| \nabla \times ({\vec E}_0 - {\vec F})\|_{L^2(B_\alpha)} .  \label{eqn:e0mfmphiorg}
\end{align}
Additionally,
\begin{align}
\left \| {\vec E}_\alpha - {\vec E}_0 - {\vec w} -  \nabla \phi_0 \right \|_{L^2(B_\alpha)} \le C \alpha 
 \| \nabla \times (  {\vec E}_\alpha - {\vec E}_0 - {\vec w}  ) \|_{L^2(B_\alpha)} . \label{eqn:orgealphae0wphi0}
\end{align}
Using the above, Ammari {\it et al.}~\cite{Ammari2014} obtain the following for $\sigma_s=0, \mu_{r,s}=1$.
\begin{lemma}[Ammari {\it et al.}~\cite{Ammari2014}] \label{lemma:orgengyest}
Let ${\vec w}$ be the solution to (\ref{eqn:wprob})  for $\sigma_s=0, \mu_{r,s}=1$. Then there exists a constant $C$ such that
\begin{subequations}
\begin{align}
\left \| \nabla \times \left ( {\vec E}_\alpha - {\vec E}_0 - {\vec w}   \right) \right \|_{L^2(B_\alpha)} \le & C \alpha^{7/2}
 \mu_{r,*} ( | 1- \mu_{r,*}^{-1} | +\nu)    \| \nabla \times {\vec E}_0 \|_{W^{2,\infty}(B_\alpha)}, \\
 \left \| {\vec E}_\alpha - {\vec E}_0 - {\vec w} -  \nabla \phi_0 \right \|_{L^2(B_\alpha)} \le & C
  \alpha^{9/2}
 \mu_{r,*} (  |1 - \mu_{r,*} ^{-1} | +\nu)  \| \nabla \times {\vec E}_0 \|_{W^{2,\infty}(B_\alpha)} .
 \end{align}
 \end{subequations}
\end{lemma}
}
\subsection{An energy estimate for $\sigma_s \ne 0$}

 {If instead $\sigma_s\ne 0$, then from (\ref{eqn:wprob}), we have 
 \begin{subequations}
\begin{align}
\nabla \cdot \sigma_s {\vec w} &= 0 && \text{in $\Omega_s \setminus \overline{B}_\alpha$} ,\\
\nabla \cdot \sigma_* {\vec w} & = \nabla \cdot( \sigma_s - \sigma_*) {\vec F} && \text{in ${B}_\alpha$} ,\\
 {{\vec n} \cdot \sigma_* {\vec w}^- } &  =  {  {\vec n} \cdot  \left ( {\sigma_s - \sigma_*} \right ) {\vec F} } &&  {  \text{on $\Gamma_*$} }.
\end{align}
\end{subequations}
Again by choosing $\nabla \cdot( {\vec F} - \nabla \phi_0 ) =0$ in $B_\alpha$ , where  explicitly, up to a constant,
\begin{align}
\phi_0 = -\frac{\im \omega \mu_0 \sigma_s }{6} ( ({\vec E}_0({\vec z}))_1 (x_1-z_1)^3 +({\vec E}_0({\vec z}))_2 (x_2-z_2)^3 + ({\vec E}_0({\vec z}))_3 (x_3-z_3)^3) , \nonumber
\end{align}
and setting
${\vec n} \cdot  ( \nabla \phi_0 +{\vec E}_0)^-  ={\vec n} \cdot  {\vec F}  ({\vec x})$ on $\Gamma_*$ and $\int_{B_\alpha} \phi_0 \dif {\vec y} =0 $, we have  $ \nabla \cdot( (\sigma_* - \sigma_s) / \sigma_*)  ( {\vec F} - \nabla \phi_0 ) =0$ and, hence, $\nabla \cdot ( - {\vec w} - ((\sigma_* - \sigma_s)/ \sigma_* \nabla \phi_0 )) =0$ in $B_\alpha$ similar to~\cite{Ammari2014,LedgerLionheart2024}. 
Also, similar to~\cite{Ammari2014},
from (\ref{eqn:we0}) we have $\nabla \cdot {\vec E}_0 =0$ in $\Omega_s$ and ${\vec n} \cdot {\vec E}_0^- =0 $ on $\Gamma_s$, and, from (\ref{eqn:wealpha}),  $\nabla \cdot {\vec E}_\alpha  = 0 $ in $\Omega_s$ and ${\vec n} \cdot {\vec E}_\alpha^- =0 $ on $\Gamma_s$. This means that, since $B_\alpha \subset \Omega_s$, }
 {\begin{align}
\| {\vec E}_0 - {\vec F} + \nabla \phi_0\|_{L^2(B_\alpha)} \le & C \alpha \left (  \| \nabla \times ({\vec E}_0 - {\vec F})\|_{L^2(B_\alpha)} + \| {\vec n} \cdot ({\vec E}_0 - {\vec F} + \nabla \phi_0 )\|_{L^2(\Gamma_* )}  \right ) \nonumber\\
\le & C \alpha   \| \nabla \times ({\vec E}_0 - {\vec F})\|_{L^2(B_\alpha)} . \label{eqn:e0mfmphi}
\end{align}
But, instead of (\ref{eqn:orgealphae0wphi0}) we have 
\begin{align}
\left \| {\vec E}_\alpha - {\vec E}_0 - {\vec w} - \left ( \frac{\sigma_*- \sigma_s}{\sigma_*} \right )  \nabla \phi_0 \right \|_{L^2(B_\alpha)} \le C \alpha 
\left ( \| \nabla \times (  {\vec E}_\alpha - {\vec E}_0 - {\vec w}  ) \|_{L^2(B_\alpha)} + \left \| {\vec n} \cdot \left ({\vec E}_\alpha -\frac{\sigma_s}{\sigma_*}  {\vec E}_0 \right ) \right  \|_{L^2(\Gamma_*)}
 \right ),\label{eqn:ealphame0mwmphi}
\end{align}
since ${\vec n}\cdot {\vec w}^- = {\vec n} \cdot \left ( \frac{\sigma_s-\sigma_*}{\sigma_*}  \right ) {\vec F} $ on $\Gamma_*$
and  where ${\vec n} \cdot {\vec E}_\alpha^- = {\vec n} \cdot \frac{\sigma_s}{\sigma_*}  {\vec E}_0^-=0$ on $\Gamma_*$ if $\sigma_s=0$. }

 {For the time being, we assume that 
 \begin{align}
  \left \| {\vec n} \cdot \left ({\vec E}_\alpha -\frac{\sigma_s}{\sigma_*}  {\vec E}_0 \right ) \right  \|_{L^2 (\Gamma_* ) } \le  \| \nabla \times (  {\vec E}_\alpha - {\vec E}_0 - {\vec w}  ) \|_{L^2(B_\alpha)}, \label{eqn:oldcond}
  \end{align}
 which we will replace with a condition of the parameters in the next section. Using this assumption, we obtain the following estimate for buried objects.}

\begin{lemma} \label{lemma:energy1}
Let ${\vec w}$ be the solution to (\ref{eqn:wprob}). Then, by  assuming that (\ref{eqn:oldcond}) holds,  there exists a constant $C$ such that
\begin{subequations}
\begin{align}
\left \| \nabla \times \left ( {\vec E}_\alpha - {\vec E}_0 - {\vec w}   \right) \right \|_{L^2(B_\alpha)} \le & C \alpha^{7/2}
 \mu_{r,*} ( | \mu_{r,s}^{-1} - \mu_{r,*}^{-1} | +\nu)    \| \nabla \times {\vec E}_0 \|_{W^{2,\infty}(B_\alpha)}, \label{oldlemma_engy1} \\
 \left \| {\vec E}_\alpha - {\vec E}_0 - {\vec w} - \left ( \frac{\sigma_*- \sigma_s}{\sigma_*} \right )  \nabla \phi_0 \right \|_{L^2(B_\alpha)} \le & C
  \alpha^{9/2}
 \mu_{r,*} (  | \mu_{r,s}^{-1} - \mu_{r,*} ^{-1} | +\nu)  \| \nabla \times {\vec E}_0 \|_{W^{2,\infty}(B_\alpha)}  \label{oldlemma_engy2}.
 \end{align}
 \end{subequations}
\end{lemma}
\begin{proof}
First we construct
\begin{align}
( & \mu_0^{-1} \nabla \times ( {\vec E}_\alpha - {\vec E}_0 -{\vec w}),  \nabla \times {\vec v} )_{\Omega_{fs} }+
(  \mu_0^{-1} \mu_{r,s}^{-1} \nabla \times ( {\vec E}_\alpha - {\vec E}_0 - {\vec w}) , \nabla \times {\vec v} )_{\Omega_{s}\setminus \overline{B_\alpha} }\nonumber\\
&+
(  \mu_0^{-1} \mu_{r,*}^{-1} \nabla \times ( {\vec E}_\alpha - {\vec E}_0-{\vec w}) ,  \nabla \times {\vec v})_{B_\alpha } 
- \im \omega \left  (  \sigma_s  \left ( {\vec E}_\alpha - {\vec E}_0-{\vec w}  \right ) , {\vec v} \right )_{\Omega_s \setminus \overline{B_\alpha} }\nonumber\\
&- \im \omega \left (  \sigma_*  \left  ( {\vec E}_\alpha - {\vec E}_0 -{\vec w} -  \left ( \frac{\sigma_* - \sigma_s}{\sigma_*} \right ) \nabla \phi_0 \right ), {\vec v}\right  )_{  {B_\alpha} } \nonumber \\
& = \mu_0^{-1} \left ( (\mu_{r,s}^{-1} - \mu_{r,*}^{-1})  \nabla \times ( {\vec E}_0 - {\vec F}) ,  \nabla \times {\vec v} \right )_{B_\alpha}
-\im \omega ( (\sigma_s-\sigma_*) ( {\vec E}_0 - {\vec F}+ \nabla \phi_0 )  , {\vec v} )_{  {B_\alpha} } , \nonumber
\end{align}
then we choose ${\vec v} = \left \{ \begin{array}{ll} {\vec E}_\alpha - {\vec E}_0 - {\vec w} -  \left ( \frac{\sigma_* - \sigma_s}{\sigma_*} \right ) \nabla \phi_0  & \text{in $B_\alpha$} \\ {\vec E}_\alpha - {\vec E}_0 - {\vec w}  & \text{in ${\mathbb R}^3 \setminus \overline{B_\alpha}$} \end{array} \right .$ and multiply by $\mu_0$ so that
\begin{align}
&  \| \nabla \times ( {\vec E}_\alpha - {\vec E}_0 -{\vec w}) \|^2_{L^2( \Omega_{fs} ) } +  \mu_{r,s}^{-1} \| \nabla \times ( {\vec E}_\alpha - {\vec E}_0 -{\vec w}) \|^2_{L^2( \Omega_{s}\setminus\overline{B_\alpha}  ) } \nonumber\\
& +  \mu_{r,*}^{-1}  \| \nabla \times ( {\vec E}_\alpha - {\vec E}_0 -{\vec w}) \|^2_{L^2( B_\alpha  ) }
- \im \omega \mu_0  \sigma_s  \left  \|    {\vec E}_\alpha - {\vec E}_0-{\vec w}  \right \|^2 _{L^2(\Omega_s \setminus \overline{B_\alpha}) }\nonumber\\
& - \im \omega\mu_0  \sigma_*  \left \|    {\vec E}_\alpha - {\vec E}_0 -{\vec w} -  \left ( \frac{\sigma_* - \sigma_s}{\sigma_*} \right ) \nabla \phi_0 \right \|^2_{  L^2({B_\alpha}) } \nonumber\\
& =  (\mu_{r,s}^{-1} - \mu_{r,*}^{-1}) \left  ( \nabla \times ( {\vec E}_0 - {\vec F}) ,  \nabla \times {\vec v} \right )_{B_\alpha}
-\im \omega \mu_0  ( (\sigma_s-\sigma_*) ( {\vec E}_0 - {\vec F}+ \nabla \phi_0)  , {\vec v} )_{  {B_\alpha} } , \nonumber
\end{align}
and  by the Cauchy-Schwartz inequality
\begin{align}
  \mu_{r,*}^{-1}& \left  \| \nabla \times \left ( {\vec E}_\alpha - {\vec E}_0 -{\vec w} - \left ( \frac{\sigma_* - \sigma_s}{\sigma_*} \right ) \nabla \phi_0  \right ) \right \|^2_{L^2( B_\alpha  ) } = \mu_{r,*}^{-1}  \| \nabla \times ( {\vec E}_\alpha - {\vec E}_0 -{\vec w}) \|^2_{L^2( B_\alpha  ) } \nonumber \\
  & \le \left | \left  ( (\mu_{r,s}^{-1} - \mu_{r,*}^{-1}) \nabla \times ( {\vec E}_0 - {\vec F}) ,  \nabla \times {\vec v} \right )_{B_\alpha} \right |
+ \omega \mu_0 \left |  ( (\sigma_s-\sigma_*) ( {\vec E}_0 - {\vec F}+ \nabla \phi_0 )  , {\vec v} )_{  {B_\alpha} }  \right |. \nonumber
\end{align}
Now
\begin{align}
\left |  (\mu_{r,s}^{-1} - \mu_{r,*}^{-1})\left ( \nabla \times ( {\vec E}_0 - {\vec F}) ,  \nabla \times {\vec v} \right )_{B_\alpha} \right | & \le  | \mu_{r,s}^{-1} - \mu_{r,*}^{-1} | \| \nabla \times ( {\vec E}_0 - {\vec F})\|_{L^2(B_\alpha)} 
\|\nabla \times {\vec v} \|_{L^2(B_\alpha)}  \nonumber \\
& \le C | \mu_{r,s}^{-1} - \mu_{r,*}^{-1} |\alpha^{3/2} \alpha^2 \| \nabla \times {\vec E}_0 \|_{W^{2,\infty}(B_\alpha)} \|\nabla \times {\vec v} \|_{L^2(B_\alpha)}  \nonumber \\
& \le C | \mu_{r,s}^{-1} - \mu_{r,*}^{-1} |\alpha^{7/2} \| \nabla \times {\vec E}_0 \|_{W^{2,\infty}(B_\alpha)} \|\nabla \times {\vec v} \|_{L^2(B_\alpha)},  \nonumber 
\end{align}
and
\begin{align}
\omega \mu_0 \left |  ( (\sigma_s-\sigma_*) ( {\vec E}_0 - {\vec F}+ \nabla \phi_0 )  , {\vec v} )_{  {B_\alpha} }  \right | \le & \omega \mu_0 | \sigma_s-\sigma_*|  \| {\vec E}_0 - {\vec F}+ \nabla \phi_0 \|_{L^2(B_\alpha)} \| {\vec v} \|_{L^2(B_\alpha)} \nonumber\\
\le & C \omega \mu_0 | \sigma_s-\sigma_*|   \alpha \| \nabla \times ({\vec E}_0 - {\vec F} ) \|_{L^2(B_\alpha)} \alpha  \| \nabla \times {\vec v} \|_{L^2(B_\alpha)}  \nonumber\\
\le & C\alpha^2  \omega \mu_0 | \sigma_s-\sigma_*|   \alpha^{3/2} \alpha^2  \| \nabla \times {\vec E}_0 \|_{W^{2,\infty}(B_\alpha)}   \| \nabla \times {\vec v} \|_{L^2(B_\alpha)}  \nonumber\\
\le & C\alpha^{11/2}  \omega \mu_0 | \sigma_s-\sigma_*|  \| \nabla \times {\vec E}_0 \|_{W^{2,\infty}(B_\alpha)}  \| \nabla \times {\vec v} \|_{L^2(B_\alpha)}  \nonumber\\
\le & C\alpha^{7/2} \nu    \| \nabla \times {\vec E}_0 \|_{W^{2,\infty}(B_\alpha)}    \| \nabla \times {\vec v} \|_{L^2(B_\alpha)},  \nonumber
\end{align}
which follows from (\ref{eqn:e0mfmphi}),
\begin{align} 
  \| \nabla \times ({\vec E}_0 - {\vec F} ) \|_{L^2(B_\alpha)} \le C \alpha^{3/2}  \alpha^3 \| \nabla \times {\vec E}_0 \|_{W^{2,\infty}(B_\alpha)},\label{eqn:e0minFminphi}
 \end{align}
 and $ \| {\vec v} \|_{L^2(B_\alpha)} \le C  \alpha  \| \nabla \times {\vec v} \|_{L^2(B_\alpha)} $ and $\alpha^2 \omega \mu_0 | \sigma_s-\sigma_*| \le \nu $.
Thus,
\begin{align}
  & \left  \| \nabla \times \left ( {\vec E}_\alpha - {\vec E}_0 -{\vec w} - \left ( \frac{\sigma_* - \sigma_s}{\sigma_*} \right ) \nabla \phi_0  \right ) \right \|^2_{L^2( B_\alpha  ) }\le
 C \mu_{r,*} \alpha^{7/2} (   | \mu_{r,s}^{-1} - \mu_{r,*}^{-1} | + \nu )\nonumber \\
& \qquad \qquad \qquad \qquad \qquad \qquad \qquad \qquad \qquad \qquad \qquad \qquad     \| \nabla \times {\vec E}_0 \|_{W^{2,\infty}(B_\alpha)}  \| \nabla \times {\vec v} \|_{L^2(B_\alpha)},
\end{align} 
which leads to (\ref{oldlemma_engy1}). Using (\ref{eqn:ealphame0mwmphi})  gives (\ref{oldlemma_engy2}).
\end{proof}

\subsection{Improved energy estimates for $\sigma_s \ne 0$} \label{sect:impenergyest}
 {To circumvent the  assumption (\ref{eqn:oldcond}), a different approach is followed below. This is motivated by considering the soil as a form of regularisation to the non-conducting background case. We illustrate this for ${\vec E}_0$ below. Defining {the bilinear forms}
\begin{align}
a ( {\vec u},{\vec v}) := & ( \nabla \times {\vec u} , \nabla \times {\vec v} )_{\Omega_{fs} \cup \Omega_{s}} ,  \nonumber \\
a^\varepsilon ({\vec u},{\vec v}) := & ( \nabla \times {\vec u} , \nabla \times {\vec v})_ {\Omega_{fs}} + ( \mu_{r,s}^{-1} \nabla \times {\vec u} , \nabla \times {\vec v} )_{ \Omega_{s}  }  - \im \omega \mu_0  ( \sigma_s {\vec u}, {\vec v} )_{ \Omega_s} , \nonumber 
\end{align}
then denoting ${\vec E}_0^\varepsilon \in \tilde{X}$ and  ${\vec E}_0^0 \in  \tilde{X}$ as the solutions to
\begin{align}
a^\varepsilon({\vec E}_0^\varepsilon,{\vec v})  = \im \omega \mu_0 f({\vec v}),  \qquad \text{and} \qquad a ({\vec E}_0^0,{\vec v})  = \im \omega \mu_0 f({\vec v}), \nonumber
\end{align}
with and without soil, respectively, where} {the linear form}  { $f( {\vec v}) :=( {\vec J}^s , {\vec v})_{\Omega_{fs}} $,
we can estimate
\begin{align}
|a ({\vec u},{\vec v})-  a^\varepsilon ({\vec u},{\vec v})| \le & |  1-\mu_{r,s}^{-1} | | (\nabla \times {\vec u}, \nabla \times {\vec v} ) _{\Omega_{s}} +\omega  \mu_0 ( \sigma_s {\vec u}, {\vec v} )_{ \Omega_s}| \nonumber \\
\le & |  1-\mu_{r,s}^{-1} | \| \nabla \times {\vec u}\|_{L^2( \Omega_{s})} \|  \nabla \times {\vec v} \| _{L^2( \Omega_{s}) } +\omega  \mu_0  \sigma_s \| {\vec u}\|_{L^2(  \Omega_s)}
\|  {\vec v} \|_{L^2(  \Omega_s) }  \nonumber .
\end{align}
Since $\tilde{X}$ involves a divergence constraint that can be imposed through a mixed formulation, we have following~\cite{zaglmayrphd}[pg 30-31], \cite{bachlnger2025},
\begin{align}
\| {\vec E}_0^0 - {\vec E}_0^\varepsilon \|_{Y} \le 
C \inf_{{\vec v} \in \tilde{X}} \left ( \|{\vec E}_0^0 - {\vec v}\|_{Y} +\sup_{{\vec w}\in \tilde{X}} \frac{|a ({\vec v},{\vec w})-  a^\varepsilon ({\vec v},{\vec w})|}{
\| {\vec w}\|_Y}
\right), \nonumber
\end{align}
where $Y= {\vec H}(\hbox{curl}, \Omega_s \cup \Omega_{fs})$. Noting ${\vec E}_0^0 \in  \tilde{X}$, we get
\begin{align}
\| {\vec E}_0^0 - {\vec E}_0^\varepsilon \|_{Y} \le C \varepsilon \| {\vec E}_0^0 \|_{Y} , \nonumber
\end{align}
where
 {$\varepsilon = \omega \mu_0 \sigma_sD^2$  (note that $|1- \mu_{r,s}^{-1}|\omega \mu_0 \sigma_sD^2 \le \varepsilon$ since $1\le \mu_{r,s} \le 1+\frac{\alpha}{D}$) and $D$ is an appropriate length scale, which is introduced so that $C$ is dimensionless.}
As ${\vec E}_0^0$ is the unique solution to $a ({\vec E}_0^0,{\vec v})  = \im \omega \mu_0 ( {\vec J}^s , {\vec v})_{\Omega_{fs}},$ we have, via a stability estimate~\cite{zaglmayrphd}[pg 30-31], \cite{bachlnger2025}, 
\begin{align}
\| \nabla \times ({\vec E}_0^0 - {\vec E}_0^\varepsilon) \|_{ L^2 (  \Omega_s ) } \le \| {\vec E}_0^0 - {\vec E}_0^\varepsilon \|_Y \le C \varepsilon \omega \mu_0 \| f\|_{ Y'}, \nonumber
\end{align}
where $C$ is independent of $\varepsilon$. Note that $\| f\|_{Y'} =\displaystyle \sup_{ {\vec v} \in Y \setminus \{ {\vec 0} \} }  \frac{ | f( {\vec v})| }{ \| {\vec v} \| _Y}$.
}

 {
In a similar way, introducing {the bilinear forms}
\begin{align}
b ( {\vec u},{\vec v}) := & ( \nabla \times {\vec u} , \nabla \times {\vec v} )_{\Omega_{fs} \cup \Omega_{s} \setminus \overline{B_\alpha}} +
( \mu_{r,*}^{-1} \nabla \times {\vec u} , \nabla \times {\vec v} )_{B_\alpha}  - \im \omega \mu_0 ( \sigma_* {\vec u} , {\vec v})_{B_\alpha} ,
 \nonumber \\
b^\varepsilon ({\vec u},{\vec v}) := & ( \nabla \times {\vec u} , \nabla \times {\vec v})_ {\Omega_{fs}} + ( \mu_{r,s}^{-1} \nabla \times {\vec u} , \nabla \times {\vec v} )_{ \Omega_{s} \setminus \overline{B_\alpha} } +  ( \mu_{r,*}^{-1} \nabla \times {\vec u} , \nabla \times {\vec v} )_{B_\alpha} \nonumber\\
&  - \im \omega \mu_0  ( \sigma_s {\vec u}, {\vec v} )_{ \Omega_s \setminus \overline{B_\alpha} }
  - \im \omega \mu_0 (\sigma_* {\vec u} , {\vec v})_{B_\alpha} ,
   \nonumber 
\end{align}
and writing ${\vec u}^\varepsilon = \left \{ \begin{array}{ll} {\vec E}_\alpha - {\vec E}_0 - {\vec w} -  \left ( \frac{\sigma_* - \sigma_s}{\sigma_*} \right ) \nabla \phi_0  & \text{in $B_\alpha$} \\ {\vec E}_\alpha - {\vec E}_0 - {\vec w}  & \text{in ${\mathbb R}^3 \setminus \overline{B_\alpha}$} \end{array} \right .$ and ${\vec u}^0 = \left \{ \begin{array}{ll} {\vec E}_\alpha - {\vec E}_0 - {\vec w} -  \nabla \phi_0  & \text{in $B_\alpha$} \\ {\vec E}_\alpha - {\vec E}_0 - {\vec w}  & \text{in ${\mathbb R}^3 \setminus \overline{B_\alpha}$} \end{array} \right .$, which satisfy
\begin{align}
b^\varepsilon({\vec u}^\varepsilon,{\vec v})  =  g({\vec v})  \qquad \text{and} \qquad b ({\vec u}^0,{\vec v})  =  g({\vec v}; \sigma_s=0, \mu_{r,s}=1) , \nonumber
\end{align}
for all ${\vec v}\in  \tilde{X}$ with and without soil, respectively. In the above, 
\begin{align}
g( {\vec v}) := \left ( (\mu_{r,s}^{-1} - \mu_{r,*}^{-1})  \nabla \times ( {\vec E}_0 - {\vec F}) ,  \nabla \times {\vec v} \right )_{B_\alpha}
-\im \omega \mu_0 ( (\sigma_s-\sigma_*) ( {\vec E}_0 - {\vec F}+ \nabla \phi_0 )  , {\vec v} )_{  {B_\alpha} }, \nonumber
\end{align}
which simplifies accordingly if $\mu_{r,s}=1$ and $\sigma_s=0$. Then, following similar steps to the above, 
we can estimate
\begin{align}
\| \nabla \times ({\vec u}^0 - {\vec u}^\varepsilon)  \|_{ L^2 (  B_\alpha ) } \le & \| {\vec u}^0 - {\vec u}^\varepsilon \|_Y  \le C \varepsilon  \| g ({\vec v}; \sigma_s=0,\mu_{r,s}=1)\|_{ Y'}  \label{eqn:regestimate}
\end{align}
 {where again  $\varepsilon  = \omega \mu_0 \sigma_sD^2 $ and $D$ corresponds to the depth of the object}. We use the above to establish the following.
\begin{lemma} \label{lemma:refenergy}
Let ${\vec w}$ be the solution to (\ref{eqn:wprob}). Provided that $\varepsilon  \le  \nu$,  there exists a constant $C$ such that
\begin{subequations}
\begin{align}
\left \| \nabla \times \left ( {\vec E}_\alpha - {\vec E}_0 - {\vec w}   \right) \right \|_{L^2(B_\alpha)} \le & C \alpha^{7/2}
 \mu_{r,*} ( | 1 - \mu_{r,*}^{-1} | +\nu)    \| \nabla \times {\vec E}_0 \|_{W^{2,\infty}(B_\alpha)}, \label{lemma_engy1} \\
 \left \| {\vec E}_\alpha - {\vec E}_0 - {\vec w} - \left ( \frac{\sigma_*- \sigma_s}{\sigma_*} \right )  \nabla \phi_0 \right \|_{L^2(B_\alpha)} \le & C
  \alpha^{9/2}
 \mu_{r,*} (  | 1 - \mu_{r,*} ^{-1} | +\nu)  \| \nabla \times {\vec E}_0 \|_{W^{2,\infty}(B_\alpha)}  \label{lemma_engy2}.
 \end{align}
 \end{subequations}
\end{lemma}
\begin{proof}
By the triangle inequality
\begin{align}
\| \nabla \times {\vec u}^\varepsilon \|_{L^2(B_\alpha) } \le  \| \nabla \times({\vec u}^\varepsilon - {\vec u}^0) \|_{L^2(B_\alpha) } +  \| \nabla \times{\vec u}^0 \|_{L^2(B_\alpha) }, \nonumber
\end{align}
 where we choose  ${\vec u}^\varepsilon = {\vec E}_\alpha - {\vec E}_0 - {\vec w} -  \left ( \frac{\sigma_* - \sigma_s}{\sigma_*} \right ) \nabla \phi_0 $ and ${\vec u}^0 = {\vec E}_\alpha - {\vec E}_0 - {\vec w} - \nabla \phi_0 $ in $B_\alpha$, and ${\vec u}^\varepsilon = {\vec E}_\alpha - {\vec E}_0 - {\vec w}  $ and ${\vec u}^0 = {\vec E}_\alpha - {\vec E}_0 - {\vec w}  $ in ${\mathbb R}^3 \setminus \overline{B_\alpha}$,    with and without the soil present, respectively. Then,  from Lemma~\ref{lemma:orgengyest} for the non-conducting soil,
 \begin{align}
 \| \nabla \times{\vec u}^0 \|_{L^2(B_\alpha) } = \left \| \nabla \times \left ( {\vec E}_\alpha - {\vec E}_0 - {\vec w}   \right) \right \|_{L^2(B_\alpha)} \le & C \alpha^{7/2} 
  \mu_{r,*} ( | 1 - \mu_{r,*}^{-1} | +\nu)    \| \nabla \times {\vec E}_0 \|_{W^{2,\infty}(B_\alpha)}, \nonumber
 \end{align}
 and from the proof of this result 
\begin{align}
 |g ({\vec v}; \sigma_s=0, \mu_{r,s}=1) | \le C \mu_{r,*} \alpha^{7/2} (   | 1 - \mu_{r,*}^{-1} | + \nu   )  \| \nabla \times {\vec E}_0 \|_{W^{2,\infty}(B_\alpha)}  \| \nabla \times {\vec v} \|_{L^2(B_\alpha)}. \nonumber
\end{align}
So, using (\ref{eqn:regestimate}) we obtain (\ref{lemma_engy1}) provided that $\varepsilon \le \nu$.  
{To obtain, (\ref{lemma_engy2}), we again use the triangle inequality
\begin{align}
\|  {\vec u}^\varepsilon \|_{L^2(B_\alpha) } \le  \|  {\vec u}^\varepsilon - {\vec u}^0 \|_{L^2(B_\alpha) } +  \|  {\vec u}^0 \|_{L^2(B_\alpha) }, \nonumber
\end{align}
where, from Lemma~\ref{lemma:orgengyest}  for the non-conducting soil,
 \begin{align}
 \|  {\vec u}^0 \|_{L^2(B_\alpha) } = \left \|  {\vec E}_\alpha - {\vec E}_0 - {\vec w}  -\nabla \phi_0  \right \|_{L^2(B_\alpha)} \le & C \alpha^{9/2} 
  \mu_{r,*} ( | 1 - \mu_{r,*}^{-1} | +\nu)    \| \nabla \times {\vec E}_0 \|_{W^{2,\infty}(B_\alpha)}. \nonumber
 \end{align}
Noting that
\begin{align}
\| {\vec u}^0 - {\vec u}^\varepsilon \|_{L^2(B_\alpha)}  & \le C \alpha \left ( \| \nabla \times  \left ( {\vec u}^0 - {\vec u}^\varepsilon \right ) \|_{L^2(B_\alpha)} + \| {\vec n} \cdot \left ({\vec u} - {\vec u}^\varepsilon \right ) \|_{L^2 (\Gamma_*)} 
\right )\nonumber\\
& \le C \alpha \| {\vec u}^0 - {\vec u}^\varepsilon \|_Y \nonumber \\
& \le C\alpha  \varepsilon  \| g ({\vec v}; \sigma_s=0,\mu_{r,s}=1)\|_{ Y'}  \nonumber,
\end{align}
we obtain the desired result provided that $\varepsilon \le \nu$.
}
\end{proof}
\begin{remark}
 {
Comparing the Lemmas~\ref{lemma:orgengyest} and ~\ref{lemma:energy1} we see they result in the effectively the same estimates with or without the soil and also justify the assumption (\ref{eqn:oldcond})  provided that $\varepsilon \le  \nu$.}
\end{remark}

\subsection{Comparison of problems for ${\vec w}$ with and without soil}
 {
We denote ${\vec w}_{fs}\in \tilde{X}$  as the solution to
\begin{align}
(  \mu_0^{-1} \nabla \times {\vec w}_{fs} ,  \nabla \times {\vec v} )_{\Omega_{fs}\cup \Omega_s \setminus \overline{B_\alpha}  }+&
(  \mu_0^{-1} \mu_{r,*}^{-1} \nabla \times {\vec w}_{fs}  ,  \nabla \times {\vec v})_{B_\alpha } - \im \omega (  \sigma_*   {\vec w}_{fs}  , {\vec v} )_{  {B_\alpha} } \nonumber \\
& = \mu_0^{-1} ( (1 - \mu_{r,*}^{-1}) \nabla \times {\vec F}  ,  \nabla \times {\vec v} )_{B_\alpha}
+\im \omega  (\sigma_*  {\vec F} , {\vec v} )_{  {B_\alpha} } \forall {\vec v} \in \tilde{X} \label{eqn:wproborg},
\end{align}
where ${\vec F}$ corresponds to the case where the soil is present. Proceeding in similar manner to before and
writing ${\vec u}^\varepsilon = \left \{ \begin{array}{ll} {\vec E}_\alpha - {\vec E}_0 - {\vec w}_{fs} -  \left ( \frac{\sigma_* - \sigma_s}{\sigma_*} \right ) \nabla \phi_0  & \text{in $B_\alpha$} \\ {\vec E}_\alpha - {\vec E}_0 - {\vec w}_{fs}  & \text{in ${\mathbb R}^3 \setminus \overline{B_\alpha}$} \end{array} \right .$ and ${\vec u}^0 = \left \{ \begin{array}{ll} {\vec E}_\alpha - {\vec E}_0 - {\vec w} -  \nabla \phi_0  & \text{in $B_\alpha$} \\ {\vec E}_\alpha - {\vec E}_0 - {\vec w}  & \text{in ${\mathbb R}^3 \setminus \overline{B_\alpha}$} \end{array} \right .$, which satisfy
\begin{align}
b^\varepsilon({\vec u}^\varepsilon,{\vec v})  =  h({\vec v}),  \qquad \text{and} \qquad b ({\vec u}^0,{\vec v})  =  h({\vec v}; \sigma_s=0, \mu_{r,s} =1 ), \nonumber
\end{align}
for all ${\vec v}\in  \tilde{X}$. In the above, 
\begin{align}
h( {\vec v}) := & \im \omega \mu_0 (\sigma_s  {\vec w}_{fs} , {\vec v})_{\Omega_s \setminus \overline{B_\alpha} } + (( 1 - \mu_{r,s}^{-1}) \nabla \times {\vec w}_{fs} , \nabla \times {\vec v} )_{\Omega_s \setminus \overline{B_\alpha}}\nonumber\\
&+ \left ( (\mu_{r,s}^{-1} - \mu_{r,*}^{-1})  \nabla \times ( {\vec E}_0 - {\vec F}) ,  \nabla \times {\vec v} \right )_{B_\alpha}
-\im \omega \mu_0 ( (\sigma_s-\sigma_*) ( {\vec E}_0 - {\vec F}+ \nabla \phi_0 )  , {\vec v} )_{  {B_\alpha} }, \nonumber
\end{align}
which simplifies accordingly if $\mu_{r,s}=1$ and $\sigma_s=0$. Then, following similar steps to the above,  we again can estimate
\begin{align}
\| \nabla \times ({\vec u}^0 - {\vec u}^\varepsilon)  \|_{ L^2 (  B_\alpha ) } \le \| {\vec u}^0 - {\vec u}^\varepsilon \|_Y  \le C \varepsilon  \| h({\vec v};\sigma_s=0,\mu_{r,s}=1)\|_{ Y'},\label{eqn:regestimate2}
\end{align}
which we use to obtain the following.
\begin{lemma}~\label{lemma:diffw}
Let ${\vec w}_{fs}$ be the solution to (\ref{eqn:wproborg}). Provided that {$\varepsilon  \le  \nu$},  there exists a constant $C$ such that
\begin{subequations}
\begin{align}
\left \| \nabla \times \left ( {\vec E}_\alpha - {\vec E}_0 - {\vec w}_{fs}   \right) \right \|_{L^2(B_\alpha)} \le & C \alpha^{7/2}
 \mu_{r,*} ( | 1 - \mu_{r,*}^{-1} | +\nu)    \| \nabla \times {\vec E}_0 \|_{W^{2,\infty}(B_\alpha)}, \label{eqn:west1} \\
 \left \| {\vec E}_\alpha - {\vec E}_0 - {\vec w}_{fs} -  \left( \frac{\sigma_*-\sigma_s}{\sigma_*} \right ) \nabla \phi_0 \right \|_{L^2(B_\alpha)} \le & C
   \alpha^{9/2}
 \mu_{r,*} (  | 1 - \mu_{r,*} ^{-1} | +\nu)  \| \nabla \times {\vec E}_0 \|_{W^{2,\infty}(B_\alpha)} 
 \label{eqn:west2}.
 \end{align}
 \end{subequations}
\end{lemma}
\begin{proof}
By the triangle inequality
\begin{align}
\| \nabla \times {\vec u}^\varepsilon\|_{L^2(B_\alpha) } \le  \| \nabla \times({\vec u}^\varepsilon - {\vec u}^0) \|_{L^2(B_\alpha) } +  \| \nabla \times{\vec u}^0 \|_{L^2(B_\alpha) }, \nonumber
\end{align}
 where we choose  ${\vec u}^\varepsilon = {\vec E}_\alpha - {\vec E}_0 - {\vec w}_{fs} -  \left ( \frac{\sigma_* - \sigma_s}{\sigma_*} \right ) \nabla \phi_0 $ and ${\vec u}^0 = {\vec E}_\alpha - {\vec E}_0 - {\vec w} - \nabla \phi_0 $ in $B_\alpha$,
and ${\vec u}^\varepsilon = {\vec E}_\alpha - {\vec E}_0 - {\vec w}  $ and ${\vec u}^0 = {\vec E}_\alpha - {\vec E}_0 - {\vec w}  $ in ${\mathbb R}^3 \setminus \overline{B_\alpha}$.
  Then,  since $g({\vec v}) = h({\vec v})$ for $\sigma_s=0$, $\mu_{r,s}=1$, noting  from the proof of Lemma~\ref{lemma:orgengyest},
\begin{align}
 |g ({\vec v}; \sigma_s=0, \mu_{r,s}=1) | \le C \mu_{r,*} \alpha^{7/2} (   | 1 - \mu_{r,*}^{-1} | + \nu   )  \| \nabla \times {\vec E}_0 \|_{W^{2,\infty}(B_\alpha)}  \| \nabla \times {\vec v} \|_{L^2(B_\alpha)}, \nonumber
\end{align}
and using (\ref{eqn:regestimate2}), we obtain (\ref{eqn:west1}) provided that $\varepsilon \le \nu$. 

{To obtain (\ref{eqn:west2}), we again use the triangle inequality
\begin{align}
\|  {\vec u}^\varepsilon \|_{L^2(B_\alpha) } \le  \|  {\vec u}^\varepsilon - {\vec u}^0 \|_{L^2(B_\alpha) } +  \|  {\vec u}^0 \|_{L^2(B_\alpha) }, \nonumber
\end{align}
where, from Lemma~\ref{lemma:orgengyest}  for the non-conducting soil,
 \begin{align}
 \|  {\vec u}^0 \|_{L^2(B_\alpha) } = \left \|  {\vec E}_\alpha - {\vec E}_0 - {\vec w}  -\nabla \phi_0  \right \|_{L^2(B_\alpha)} \le & C \alpha^{9/2} 
  \mu_{r,*} ( | 1 - \mu_{r,*}^{-1} | +\nu)    \| \nabla \times {\vec E}_0 \|_{W^{2,\infty}(B_\alpha)}, \nonumber
 \end{align}
and noting that
\begin{align}
\| {\vec u}^0 - {\vec u}^\varepsilon \|_{L^2(B_\alpha)}  & \le C \alpha \left ( \| \nabla \times  \left ( {\vec u}^0 - {\vec u}^\varepsilon \right ) \|_{L^2(B_\alpha)} + \| {\vec n} \cdot \left ({\vec u} - {\vec u}^\varepsilon \right ) \|_{L^2 (\Gamma_*)} 
\right )\nonumber\\
& \le C \alpha \| {\vec u}^0 - {\vec u}^\varepsilon \|_Y \nonumber \\
& \le C\alpha  \varepsilon  \| h ({\vec v}; \sigma_s=0,\mu_{r,s}=1)\|_{ Y'}  \nonumber,
\end{align}
we obtain the desired result provided that $\varepsilon \le \nu$.
}
\end{proof}
}

\section{Transmission problems independent of the object position} \label{sect:transprob}
In a similar way to~\cite{Ammari2014}, we write ${\vec w}_{fs} = \alpha {\vec w}_{0,fs}( ({\vec x} - {\vec z})/\alpha)$ where ${\vec w}_{0,fs} ({\vec \xi})$ satisfies
\begin{subequations}\label{eqn:wo0orgprob}
\begin{align}
\nabla_\xi \times \mu_{r,s}^{-1} \nabla_\xi \times {\vec w}_{0,fs} - \im \omega \mu_0 \alpha^2  \sigma_* {\vec w}_{0,fs}  = & \im \omega \mu_0 \alpha^2 \sigma_* [ \alpha^{-1} {\vec F}( {\vec z} + \alpha {\vec \xi})] && \text{in $B$ } ,\\
\nabla_\xi \times \nabla_\xi \times {\vec w}_{0,fs}  =& {\vec 0} && \text{in $B^c$},\\
\nabla_\xi  \cdot {\vec w}_{0,fs}  =& 0 && \text{in $B^c$}, \\
[{\vec n} \times {\vec w}_{0,fs} ]_\Gamma ={\vec 0}, [ {\vec n} \times \mu_r^{-1} \nabla \times {\vec w}_{0,fs} ]_\Gamma= &-\im \omega \mu_0 \mu_{r,s} ( 1- \mu_r^{-1}) {\vec n} \times ( {\vec H}_0 ({\vec z}) + \alpha {\vec D} {\vec H}_0 ( {\vec z}) {\vec \xi}) && \text{on $\Gamma:= \partial B$}, \\
{\vec w}_{0,fs} ({\vec \xi}) = & O(|{\vec \xi}|^{-1}) && \text{as $|{\vec \xi}| \to \infty$},
\end{align}
\end{subequations}
and ${\vec \xi}$ has an origin in $B$.  It then follows from Lemma~\ref{lemma:diffw} that

 {
\begin{lemma}
Let ${\vec w}$ be the solution to and ${\vec w}_{0,fs}$ be the solution to  (\ref{eqn:wo0orgprob}). Provided that  $\varepsilon \le  \nu$, there exists a constant $C$ such that
\begin{subequations}
\begin{align}
\left \| \nabla \times \left ( {\vec E}_\alpha - {\vec E}_0 - \alpha{\vec w}_{0,fs}  \left ( \frac{{\vec x}-{\vec z}}{\alpha} \right )  \right) \right \|_{L^2(B_\alpha)} \le & C \alpha^{7/2}
 \mu_{r,*} ( | 1 - \mu_{r,*}^{-1} | +\nu)    \| \nabla \times {\vec E}_0 \|_{W^{2,\infty}(B_\alpha)},\label{thm:wdiff1} \\
 \left \| {\vec E}_\alpha - {\vec E}_0 -\alpha {\vec w}_{0,fs} \left ( \frac{{\vec x}-{\vec z}}{\alpha} \right )-  \left( \frac{\sigma_*-\sigma_s}{\sigma_*} \right )\nabla \phi_0 \right \|_{L^2(B_\alpha)} \le & C
  \alpha^{9/2}
 \mu_{r,*} (  | 1 - \mu_{r,*} ^{-1} | +\nu)  \| \nabla \times {\vec E}_0 \|_{W^{2,\infty}(B_\alpha)}\label{thm:wdiff2}.
\end{align}
\end{subequations}
\end{lemma}
}

Given that $B_\alpha \subset \Omega_s$, we have
\begin{align}
\alpha^{-1} {\vec F} ( {\vec z} + \alpha {\vec \xi}) = \im \omega  \mu_0 \mu_{r,s}
\left (\sum_{i=1}^3 \frac{1}{2} ({\vec H}_0 ({\vec z}))_i {\vec e}_i \times {\vec \xi}  +\sum_{i,j=1}^3  \frac{\alpha}{3} ({\vec D}{\vec H}_0 )_{ij} ({\vec e}_i \otimes {\vec e}_j ){\vec \xi} \times {\vec \xi}
 \right ) ,
  \label{eqn:alpham1F}
\end{align}
and similarly,
\begin{align}
{\vec w}_{0,fs}({\vec \xi})= \im \omega \mu_0 \mu_{r,s}  \left ( \sum_{i=1}^3 \frac{1}{2} ({\vec H}({\vec z}))_i {\vec \theta}_i ({\vec \xi}) + \sum_{i,j=1}^3 \frac{\alpha}{3} ( {\vec D}{\vec H}({\vec z}))_{ij} {\vec \psi}_{ij}({\vec \xi}) \right ) .
\label{eqn:expandw0fs}
\end{align}
As in (3.15) in~\cite{Ammari2014}, ${\vec \theta}_i$ satisfies
\begin{subequations} \label{eqn:thetatrans}
\begin{align}
\nabla_\xi\times  \mu_{r,*}^{-1} \nabla_\xi \times {\vec \theta}_i - \im \omega \mu_0 \alpha^2 \sigma_* {\vec \theta}_i = & \im  \omega \mu_0 \alpha^2 \sigma_* {\vec e}_i \times {\vec \xi} && \text{In $B$ },\\
\nabla_\xi \times \nabla_\xi \times {\vec \theta}_i = & {\vec 0} && \text{In $B^c$}, \\
\nabla_\xi \cdot  {\vec \theta}_i = & {\vec 0} && \text{In $B^c$}, \\
[{\vec n} \times {\vec \theta}_i]_\Gamma = {\vec 0}, [ {\vec n} \times \mu_r^{-1} \nabla_\xi \times {\vec \theta}_i ]_\Gamma = &- 2 (1-\mu_{r,*}^{-1}){\vec e}_i \times {\vec n} && \text{on $\Gamma$},  \\
{\vec \theta}_i = & O(1/|{\vec \xi} |) && \text{as $|{\vec \xi}| \to \infty$},
\end{align}
\end{subequations}
with ${\vec \psi}_{ij}$ satisfying the transmission problem stated in (3.16) in~\cite{Ammari2014}. Importantly, ${\vec \theta}_i$ is independent of $\sigma_s$ and $\mu_{r,s}$ and its solution does not depend on the  position of $B_\alpha$  in $\Omega_s$. We refer to Section 3 of~\cite{Ammari2014} for further properties regarding ${\vec \theta}_i$. 

\section{Integral representation formulae} \label{sect:integalrep}
In the previous work~\cite{Ammari2014}, an integral representation formula has been obtained to represent fields exterior to a bounded object $B_\alpha$  where the background medium corresponds to non-conducting free space with permeability $\mu_0$. In~\cite{LedgerLionheart2018} we have shown,  that since the perturbed magnetic vector potential $\Delta {\vec A}= {\vec A}_\alpha - {\vec A}_0$, where ${\vec H}_\alpha = \mu_0^{-1} \nabla \times {\vec A}_\alpha$ and ${\vec H}_0 = \mu_0^{-1} \nabla \times {\vec A}_0$ for ${\vec x} \in {\mathbb R}^3 \setminus \overline{B_\alpha}$,  satisfies $\nabla \times \mu_0^{-1} \nabla \times \Delta {\vec A} ={\vec 0}$, this representation formula can  be written in the form
\begin{align}
\Delta {\vec A} ({\vec x}) = \mu_0 \int_{B_\alpha} {\mathcal G}_0({\vec x},{\vec y}) {\vec J}_e ({\vec y}) \dif {\vec y} - \mu_0 \int_{B_\alpha} (\nabla_x \times {\mathcal G}_0({\vec x},{\vec y})  {\vec M}_e ({\vec y})) \dif {\vec y} ,
\label{eqn:apertdy}
\end{align}
for ${\vec x} \in {\mathbb R}^3 \setminus \overline{B_\alpha}$
where ${\mathcal G}_0({\vec x},{\vec y}) = G_0 ({\vec x},{\vec y}) {\mathbb I}$, $G_0 ({\vec x},{\vec y}) = \frac{1}{4 \pi | {\vec x}- {\vec y}|}$ is the free space Laplace Green's function satisfying 
\begin{align}
\nabla^2  G_0 ({\vec x},{\vec y}) = - \delta({\vec x}- {\vec y}), \nonumber
\end{align}
and ${\vec J}_e = \im \omega \sigma _*{\vec E}_\alpha$ and ${\vec M}_e = (\mu_{r,*}-1)  {\vec H}_\alpha$
 being  effective electric and magnetic source currents inside the object, respectively. The perturbed magnetic field exterior to the object $\Delta {\vec H}({\vec x}) = ({\vec H}_\alpha - {\vec H}_0)({\vec x})$ is then
\begin{align}
\Delta {\vec H}({\vec x})  & =\mu_0^{-1}  \nabla_x \times \Delta {\vec A} ({\vec x}) \nonumber\\
& =  \int_{B_\alpha}\nabla_x \times (  {\mathcal G}_0({\vec x},{\vec y}) {\vec J}_e ({\vec y}))  \dif {\vec y} - \int_{B_\alpha} \nabla_x \times \nabla_x \times (  {\mathcal G}_0({\vec x},{\vec y}) {\vec M}_e ({\vec y}))  \dif {\vec y} \label{eqn:hpertdy} \\
&=\int_{B_\alpha} \sigma_*  \nabla_x  G_0({\vec x},{\vec y}) \times {\vec E}_\alpha({\vec y}) \dif {\vec y} +   \int_{B_\alpha} (\mu_{r,*}-1) {\vec D}_x^2 G_0({\vec x}, {\vec z})  {\vec H}_\alpha({\vec y}) \dif {\vec y} ,\label{eqn:hpert}
\end{align}
since $\nabla_x \times \nabla_x \times ({ G}_0({\vec x},{\vec y}){\vec M}_e({\vec y})) =  - {\vec D}_x^2 { G}({\vec x},{\vec y}) {\vec M}_e({\vec y})$. The form of  (\ref{eqn:hpertdy}) and (\ref{eqn:apertdy}) is consistent with the those found in the literature in non-destructive testing e.g.~\cite{bowler1987} and geophysical exploration~\cite{raiche1974}, but with different Green's functions.

In our  present situation, the background medium is inhomogeneous due to the presence of the soil and consequently a different Green's function is needed. The new Green's function will be constructed to take account of the soil-air interface with the resulting dyadic Green's function ${\mathcal G}({\vec x},{\vec y})$ no longer being a multiple of identity in general and instead is the solution of 
\begin{subequations} \label{eqn:Greentrans}
\begin{align}
\nabla_x \times  \nabla_x \times {\mathcal G} ({\vec x},{\vec y})  = &  {\vec 0} && \text{in $\Omega_{fs} $},  \\
\nabla_x \times \mu_{r,s}^{-1}  \nabla_x \times {\mathcal G} ({\vec x},{\vec y}) -  \im \omega \mu_0 \sigma_s {\mathcal  G} ({\vec x},{\vec y}) = & {\vec 0} && \text{in $\Omega_s \setminus \overline{B_\alpha}$},   \\
\nabla_x \times  \mu_{r,s}^{-1} \nabla_x \times {\mathcal G} ({\vec x},{\vec y}) -  \im \omega \mu_0 \sigma_s  {\mathcal  G} ({\vec x},{\vec y}) = & -  \mu_0 \delta ({\vec x},{\vec y}) {\mathbb I} && \text{in $B_\alpha$},  \\
[{\vec n} \times {\mathcal G} ({\vec x},{\vec y}) ]_{\Gamma} = {\vec 0} , 
[{\vec n} \times \mu_{r}^{-1}\nabla_x \times {\mathcal G} ({\vec x},{\vec y}) ]_{\Gamma}  = &  {\vec 0} , && \text{on $\Gamma={\Gamma_s,\Gamma_* } $},  \\
\nabla_y \cdot  {\mathcal G} ({\vec x},{\vec y}) = & 0 && \text{in ${\mathbb R}^3 $},  \\
{\mathcal G} ({\vec x},{\vec y}) = & O(1/|{\vec y} | ) && \text{as $|{\vec y}|\to \infty$},
\end{align}
\end{subequations}
and satisfying the symmetry condition $\mu_r ({\vec y}) {\mathcal G} ({\vec x},{\vec y}) = \mu_r({\vec x}) ({\mathcal G} ({\vec y},{\vec x}) )^T$.
 Given that the Green's function will be constructed to take account of the soil-air interface,  the effective electric current in the object is ${\vec J}_e = \im \omega ( \sigma_* - \sigma_s) {\vec E}_\alpha$ and the corresponding effective magnetic current is ${\vec M}_e = (\mu_{r,*}-\mu_{r,s} )  {\vec H}_\alpha$ and we have
 \begin{align}
\Delta {\vec H}({\vec x})  & =\mu_0^{-1}  \nabla_x \times \Delta {\vec A} ({\vec x}) 
=  \int_{B_\alpha }\nabla_x \times (  {\mathcal G}({\vec x},{\vec y}) {\vec J}_e ({\vec y}))  \dif {\vec y} - \int_{B_\alpha} \nabla_x \times \nabla_x \times (  {\mathcal G}({\vec x},{\vec y}) {\vec M}_e ({\vec y}))  \dif {\vec y}, \nonumber
\end{align}
for some appropriate Green's function $ {\mathcal G}({\vec x},{\vec y})$ with ${\vec x} \in \Omega_{fs}$.

Fixing $\Gamma_s$ to be the plane $x_3=0$, Raiche and Coggon~\cite{raiche1975} consider the construction of dyadic Green's function for the perturbed electric and magnetic   fields for the case where $\mu_{r,s}=\mu_{r,*}=1$
\begin{subequations}
\begin{align}
\Delta {\vec E} ({\vec x}) =& \im \omega \mu_0 \int_{target} {\mathcal G}_E({\vec x},{\vec y}) {\vec J}_e({\vec y}) \dif {\vec y} \label{eqn:Efildgeenbowler} , \\
\Delta {\vec H} ({\vec x}) =&  \int_{target} \nabla_x \times ({\mathcal G}_E({\vec x},{\vec y}) {\vec J}_e({\vec y}) ) \dif {\vec y} \label{eqn:Hfildgeenbowler} ,
\end{align}
\end{subequations}
for a  related eddy current problem in geophysics~\cite{raiche1974} and have obtained an explicit expression for ${\mathcal G} ({\vec x},{\vec y})$ when  ${\vec x} =(0,0,x_3) \in \Omega_s$~\cite{raiche1975}. It is also related to the work of Bowler {\it et. al}~\cite{bowler1987,bowler1991}, who considers a problem with  half spaces for  air  and the conductor and, given the presence of a small volumetric flaw in the conductor with  ${\vec x}$ in the conductor, they find~\cite{bowler1991}
\begin{subequations}
\begin{align}
\Delta {\vec E} ({\vec x}) =& \im \omega \mu_0 \int_{flaw} {\mathcal G}_E({\vec x},{\vec y}) {\vec P}({\vec y}) \dif {\vec y} \label{eqn:bowlerscat} , \\
\Delta {\vec H} ({\vec x}) =& \int_{flaw} \nabla \times ({\mathcal G}_E({\vec x},{\vec y}) {\vec P}({\vec y})) \dif {\vec y} ,
\end{align}
\end{subequations}
where ${\vec P}$ is an induced current dipole density due to a non-conducting flaw within the conductor (again assuming the conductor and flaw have permeability $\mu_0$). Explicitly~\cite{bowler1991}
\begin{align}
{\mathcal G}_E({\vec x},{\vec y}) = {\mathcal G}_k  ({\vec x},{\vec y})  + {\mathcal G}_i ({\vec x},{\vec y} ) + \frac{1}{k^2} \nabla_x \times   {\vec e}_3 \nabla_y \times (V ({\vec x}, {\vec y}){\vec e}_3), \label{eqn:bowlergreen}
\end{align}
where
\begin{align}
 {\mathcal G}_k  ({\vec x},{\vec y}) &:= \left ( {\mathbb I} + \frac{1}{k^2}  \nabla_x \nabla_x  \right )  G_k ({\vec x},{\vec y} ) ,  \nonumber \\
 {\mathcal G}_i  ({\vec x},{\vec y})  &: = \left ( {\mathbb I} \myprime- \frac{1}{k^2}  \nabla_x \nabla_y  \right )  G_k ({\vec x},{\vec y}\myprime) , \nonumber \\
  G_k ({\vec x},{\vec y}) &:=  \frac{e^{\im k | {\vec x} - {\vec y}| }}{ 4 \pi | {\vec x} - {\vec y} | } , \nonumber \\
  V &:= \frac{k^2}{(2\pi)^2 }  \int_{-\infty}^\infty \int_{-\infty}^\infty \frac{1}{\kappa \gamma( \kappa + \gamma)} e^{- \gamma |x_3 + y_3| + \im u ( x_1-y_1) + \im v(x_2-y_2)} \dif u \dif v , \nonumber
\end{align}
with $k=- \im \omega \mu_0 \sigma_s$, ${\vec y}\myprime = {\vec y} -2 y_3 {\vec e}_3$,  ${\mathbb I}\myprime = \text{diag} (1,1,-1)$,  $\kappa= \sqrt{u ^2+v^2} $ and $\gamma=\sqrt{\kappa^2-k^2}$. In the case of surface cracks, Bowler and Harfield obtain a very similar result for ${\mathcal G}_E({\vec x},{\vec y})$ for $\mu_{r,s} \ne 1$ when considering ${\vec x}$ in the conducting host~\cite{bowler1998}.
 While related, these results do not provide the Green's function we require since our interest lies in evaluation of $ \Delta{\vec H}$  for ${\vec x}\in \Omega_{fs}$ and consequently we need an expression for the Green's function valid for evaluation for ${\vec x}$ in this region. Note that given the soil properties, we also expect  ${\mathcal G}({\vec x},{\vec y})$ to be close to ${\mathcal G}_0({\vec x},{\vec y})= G_0({\vec x}, {\vec y}) {\mathbb I}$ . The Green's function we require is obtained in the following Lemma, which follows similar steps to the derivations by Bowler~\cite{bowler1987, bowler2004}.

\begin{lemma}
The solution ${\mathcal G}( {\vec x},{\vec y} )$ to 
(\ref{eqn:Greentrans}) at positions ${\vec x}\in\Omega_{fs}$, when $\Gamma_s$ is the surface $x_3=0$, is of the form
\begin{align}
{\mathcal G} = G_s({\vec x},{\vec y}) {\mathbb I}, 
\end{align}
where
\begin{align}
 G_s({\vec x},{\vec y})  = - \nabla_t^2 U\myprime=  & \frac{1}{(2\pi)^2} \int_{- \infty}^{\infty} \int_{-\infty}^{\infty} \frac{ 1 } {  \gamma+ \mu_{r,s} \kappa } e^{-\gamma x_3 +\kappa y_3 + \im u (x_1-y_1) + \im v ( x_2- y_2)  }\dif u \dif v, \label{eqn:greengs}
\end{align}
for $x_3>0$ with $\kappa= \sqrt{u^2+v^2}$,  $\gamma = \sqrt{\kappa^2 -k^2}$, taking roots with positive real parts and  $k^2 = \im \omega \mu_0 \sigma_s$.
\end{lemma}
\begin{proof}
We follow similar steps to~\cite{bowler2004} and focus on determining the field due to an effective source current ${\vec J}_e$ with support in $B_\alpha \subset\Omega_s$. Following their notation, we  denote the soil region as $\Omega_0$ and the air region as $\Omega_1$ (equivalent to $\Omega_s$ and $\Omega_{fs}$, respectively) with $k_0^2 = - \im \omega \mu_0 \sigma_s$ and $k_1 =0$ and fix the surface $\Gamma_s$ as $x_3=0$. The electric and magnetic fields can then be expressed in terms of TE and TM terms as
\begin{align}
\Delta{\vec E}({\vec x})  =& -\im \omega \mu_0 \mu_{r,i} \left ( \nabla \times ( \psi\myprime ({\vec x})  {\vec e}_3 ) + \frac{1}{\mu_{r,i} k_0^2} \nabla \times \nabla \times ( \psi \mydprime ({\vec x}) ){\vec e}_3 \right ) , \nonumber\\
\Delta {\vec H} ({\vec x})  =&\nabla \times \nabla \times ( \psi\myprime ({\vec x})  {\vec e}_3 ) + \frac{k_i^2}{\mu_{r,i} k_0^2} \nabla \times ( \psi \mydprime ({\vec x}) {\vec e}_3) , \nonumber
\end{align} 
where the indiex $i$ refer to ${\vec x}\in\Omega_i$ and with the TM term vanishes for the magnetic field in air.
The potentials can be represented as
\begin{align}
\psi\myprime ({\vec x}) = \int_{B_\alpha} \nabla_y \times ( U\myprime ({\vec x},{\vec y}) {\vec e}_3) \cdot {\vec J}_e({\vec y}) \dif {\vec y} , \nonumber\\
\psi\mydprime ({\vec x}) = \int_{B_\alpha} \nabla_y \times \nabla_y \times ( U\mydprime ({\vec x},{\vec y}) {\vec e}_3) \cdot {\vec J}_e({\vec y}) \dif {\vec y} \nonumber ,
\end{align}
and  satisfy the interface conditions 
\begin{align}
[ \mu_r \psi \myprime ] & =0, && [\sigma \psi \mydprime ] =0, \nonumber \\
\left [\frac{ \partial  \psi \myprime}{\partial x_3} \right  ] & =0, &&\left  [ \frac{\partial \psi \mydprime}{\partial x_3} \right  ] =0,  \nonumber
\end{align}
on $x_3=0$,
which follow from $[{\vec n} \times \Delta{\vec E}] = [{\vec n} \times \Delta{ \vec H} ] = {\vec 0}$ on $\Gamma_s$ with ${\vec n} ={\vec e}_3$ and using
\begin{align}
{\vec e}_3 \times \nabla \times ( \psi {\vec e}_3) =&  \nabla_t \psi , \nonumber \\
{\vec e}_3 \times \nabla \times \nabla \times  ( \psi {\vec e}_3)  =  & - \frac{\partial}{\partial x_3} \nabla_t \times ( \psi {\vec e}_3). \nonumber
\end{align}
We  look for Green function representations of the form 
\begin{align}
g ({\vec x}, {\vec y})  =  \nabla_t^2 U({\vec x},{\vec y}), \nonumber
\end{align}
with different representations for ${\vec x}\in \Omega_0$ and ${\vec x}\in \Omega_{1}$.
Continuing to follow~\cite{bowler2004}, solutions of the form
\begin{align}
(\nabla^2 +k^2 ) g({\vec x},{\vec y}) = \left \{ \begin{array}{cc} 0 & \text{$x_3 > 0$}\\ -\delta ({\vec x}-{\vec y}) & \text{$x_3 < 0$} \end{array} \right . \nonumber ,
\end{align}
where $k=k_0$ subject to $[\alpha g]_{\Gamma_s}=0 $ and $[\partial g/ \partial x_3]_{\Gamma_s}=0$ on $x_3=0$ are sought leading to two dimensional Fourier transformed Greens functions
\begin{subequations}
\begin{align}
\tilde{g} ( \kappa, x_3 ,y_3) & = \frac{1}{2 \gamma} \left ( e^{- \gamma  |x_3 - y_3 |} + \Gamma e^{\gamma (x + y_3 )} \right ) && x_3 < 0, \\
\tilde{g} ( \kappa, x_3 ,y_3 ) & = \frac{1}{2 \gamma} T  e^{- \kappa x_3 + \gamma y_3  } && x_3 > 0 , \label{eqn:Uzpos}
 \end{align}
 \end{subequations}
 where 
 \begin{align}
 g({\vec x},{\vec y}) = \frac{1}{(2\pi)^2} \int_{-\infty}^\infty  \int_{-\infty}^\infty  \tilde{g}( \kappa,x_3,y_3) e^{\im u (x_1-y_1) + \im v ( x_2- y_2) } \dif u \dif v , \label{eqn:greengsint}
 \end{align}
 and $\kappa= \sqrt{u^2+v^2}$, $u,v$ are Fourier variables, and $\gamma = \sqrt{\kappa^2 -k^2}$ (roots with real parts). For the TE mode
 \begin{align}
 \Gamma \myprime= \frac{ \gamma - \mu_{r,0}\kappa}{ \gamma+\mu_{r,0}  \kappa}, \qquad T\myprime = \frac{2  \gamma}{ \gamma+ \mu_{r,0} \kappa}, \nonumber
 \end{align}
 and for the TM mode
 \begin{align}
 \Gamma \mydprime =-1, \qquad T\mydprime = 0, \nonumber
 \end{align}
 which allows the construction of the functions $U\myprime ({\vec x},{\vec y} )$ and $U\mydprime ({\vec x},{\vec y} )$ and, hence, $\Delta {\vec E}  ({\vec x})$ for ${\vec x} \in \Omega_0$. This in turn can be written in the form of (\ref{eqn:Efildgeenbowler}) and leads to ${\mathcal G}_E$ stated in (\ref{eqn:bowlergreen}) (see also Remark~\ref{remark:noncondgreen} below, which deals with the case where $\Omega_0$ (i.e $\Omega_s$) reduces to free space).
 On the other hand, considering  $\Delta {\vec H} ({\vec x})$   for ${\vec x} \in \Omega_{1} $ (i.e. ${\vec x} \in \Omega_{fs} $) using (\ref{eqn:Uzpos}) and the integral (\ref{eqn:greengsint}) we obtain
 \begin{align}
 U\myprime({\vec x},{\vec y}) = \frac{1}{(2\pi)^2} p.v. \int_{- \infty}^{\infty} \int_{-\infty}^{\infty} \frac{ 1 } { \kappa^2 ( \gamma+ \mu_{r,s} \kappa) } e^{ - \gamma x_3 + \kappa y_3 + \im u (x_1-y_1) + \im v ( x_2- y_2)  }\dif u \dif v . \nonumber
 \end{align}
 While this integral has a singularity at $\kappa=0$, this is avoided as in practice by taking tangential derivatives of  $U\myprime$. Hence, the representation of the magnetic field in terms of a scalar function $G_s ( {\vec x}, {\vec y} ) =  \nabla_t^2 U\myprime ({\vec x},{\vec y})$ stated in (\ref{eqn:greengs}), which can be computed in a similar way to~\cite{bowler1987}.
\end{proof}

\begin{remark}\label{remark:noncondgreen}
For the limiting case of a non-conducting soil, we have   $\sigma_s=0$ so that $k_0=0$ and $\gamma=\kappa$, then,  if additionally $\mu_{r,s}=\mu_{r,0}=1$ we have $\Gamma\myprime=0$ and $T\myprime=1$ so that
\begin{align}
U\myprime ({\vec x},{\vec y} ) =U\mydprime ({\vec x},{\vec y} ) = U_0 ({\vec x},{\vec y} )   = \frac{1}{(2\pi)^2} p.v. \int_{-\infty}^\infty  \int_{-\infty}^\infty   \frac{1}{2\kappa^3} e^{-\kappa (x_3-y_3)+ \im u (x_1-y_1) + \im v ( x_2- y_2) } \dif u \dif v ,
\end{align}
and
\begin{align}
G_0({\vec x},{\vec y} ) = \frac{1}{4\pi|{\vec x}- {\vec y} |} = - \nabla_t^2 U_0 =  \frac{1}{(2\pi)^2} p.v, \int_{- \infty}^{\infty} \int_{-\infty}^{\infty} \frac{1}{2\kappa } e^{- \kappa ( x_3 -  y_3) + \im u (x_1-y_2) + \im v ( x_2- y_2) }\dif u \dif v ,
\end{align}
being the Laplace Green's function~\cite{bowler2004}.  The singularity at $\kappa=0$ is avoided by taking the Cauchy principal value. For example, for ${\vec x} \in \Omega_{fs}$ and ${\vec J}_e$ with support  $B_\alpha \subset \Omega_s$,
\begin{align}
\Delta {\vec H}  ({\vec x}) & = \nabla \times \nabla \times \left ( \psi\myprime ({\vec x})  {\vec e}_3 )  \right ) \nonumber\\
& =  \nabla_x \times \nabla_x \times \left ( \int_{B_\alpha} \nabla _y \times (U\myprime ({\vec y}) {\vec e}_3) \cdot {\vec J}_e ({\vec y}) \dif {\vec y} \right ) \nonumber\\
& =   -\left ( \int_{B_\alpha} \nabla _y G_0({\vec x},{\vec y}) \times {\vec J}_e ({\vec y}) \dif {\vec y} \right ) , 
\end{align}
which, follows by applying similar arguments to~\cite{bowler1987, bowler2004} and agrees with (\ref{eqn:hpert}). 
Additionally, as $\sigma_s \to 0$ (and/or $\omega \to 0$) then $k_0 \to 0 $, and consequently $\gamma \to \kappa$ so that $U\myprime ({\vec x},{\vec y}) \to U_0 ({\vec x},{\vec y})$ and, hence, $G_s ({\vec x},{\vec y}) \to G_0({\vec x},{\vec y})$ as $k_0 \to 0$.
\end{remark}

Thus, we have for ${\vec x} \in \Omega_{fs}$,
\begin{align}
\Delta {\vec H}({\vec x})  =\int_{B_\alpha} (\sigma_*-\sigma_s)  \nabla_x G_s({\vec x},{\vec y}) \times {\vec E}_\alpha ({\vec y})  \dif {\vec y} +   \int_{B_\alpha} (\mu_{r,*}-\mu_{r,s}) {\vec D}_x^2 G_s({\vec x}, {\vec z})  {\vec H}_\alpha ({\vec y}) \dif {\vec y}, \label{eqn:hfield[ertsoil}
\end{align}
which is the starting point of the asymptotic expansion derived in the next section.

\section{Asymptotic formulae} \label{sect:asymform}

In this section we prove the following result.

\begin{theorem}
Let $\nu$ be order one and let $\alpha$ be small with $\epsilon\le \nu$. For ${\vec x} \in \Omega_{fs}$  {and objects away from $\Gamma_s$}  
\begin{align}
({\vec H}_\alpha - {\vec H}_0)({\vec x}) = & - \im \nu \alpha^3   \sum_{i=1}^3\int_{B} {\vec D}_x^2 G_s ({\vec x},{\vec z}){\vec \xi} \times ( {\vec e}_i  \times {\vec \xi} + {\vec \theta}_i ) \dif {\vec \xi} ({\vec H}_0({\vec z}))_i 
\nonumber \\
& + \left (1- \frac{1 }{ \mu_{r,*} }  \right )  \alpha^3  \sum_{i=1}^3 \int_B {\vec D}_x^2 G_s(  {\vec x}, {\vec z}) \left ( {\vec e}_i + \frac{1}{2} \nabla \times {\vec \theta}_i  \right ) \dif {\vec \xi} ({\vec H}_0({\vec z}))_i  + {\vec R}({\vec x}) ,
\label{eqn:mainresult}
\end{align}
 {where $ | {\vec R} ( {\vec x} ) | \le C \alpha^4  \|{\vec H}_0 \|_{W^{2,\infty}(B_\alpha)} $ uniformly in ${\vec x}$ in a compact set away from $B_\alpha$.}
\end{theorem}
\begin{proof}
From (\ref{eqn:hfield[ertsoil}) we can write
\begin{align}
\tilde{\vec H}_\alpha ({\vec x})= & {\vec I}_1 + {\vec I}_2 + {\vec I}_3, 
  \nonumber
\end{align}
where
\begin{align}
{\vec I}_1  &  :=  \int_{ {B_\alpha} }  \sigma_*  \nabla_x G_s ({\vec x},{\vec y}) \times {\vec E}_\alpha \dif {\vec y} ,  \nonumber\\
{\vec I}_2  &  :=  -\int_{ {B_\alpha} }  \sigma_s  \nabla_x G_s ({\vec x},{\vec y}) \times {\vec E}_\alpha \dif {\vec y} ,  \nonumber\\
{\vec I}_3  &  :=  
 -\int_{B_\alpha} \left ( \mu_{r,s} - \mu_{r,*} \right ) {\vec D}_x^2 G_s  ({\vec x},{\vec y}) {\vec H}_\alpha \dif {\vec y} ,
 \nonumber 
 \end{align}

\noindent{\bf Consider term ${\vec I}_1$} \\
We write ${\vec I}_1$ as
\begin{align}
{\vec I}_1  = {\vec I}_{1,1} +  {\vec I}_{1,2} + {\vec I}_{1,3} + {\vec I}_{1,4} , \nonumber  
\end{align} 
where
\begin{align}
{\vec I}_{1,1}  & := \int_{B_\alpha}  \sigma_* \nabla_x G_s ( {\vec x}, {\vec y}) \times \left ( {\vec E}_\alpha - {\vec E}_0 - {\vec w}_{fs} - \left ( \frac{\sigma_*-\sigma_s}{\sigma_*}\right )  \nabla \phi_0 \right ) \dif {\vec y} ,  \nonumber \\
{\vec I}_{1,2}  & := \int_{B_\alpha}  \sigma_* \nabla_x G_s ( {\vec x}, {\vec y}) \times \left ( {\vec E}_0 + \nabla \phi_0 -{\vec F} ({\vec y}) \right ) \dif {\vec y}  
-\int_{B_\alpha}  \sigma_s \nabla_x G_s ( {\vec x}, {\vec y})\times \nabla \phi_0 \dif {\vec y} ,
 \nonumber \\
{\vec I}_{1,3}  & := \int_{B_\alpha}  \sigma_* \left (  \nabla_x G_s ( {\vec x}, {\vec y}) - \nabla_x G_s ( {\vec x}, {\vec z})   + {\vec D}^2_x G_s ({\vec x}, {\vec y})  ({\vec y}- {\vec z})
\right )
 \times \left ( {\vec F} ({\vec y} ) + \alpha {\vec w}_{0, fs} \left( \frac{y-z}{\alpha} \right ) \right ) \dif {\vec y},  \nonumber \\
{\vec I}_{1,4}  & := \int_{B_\alpha} \sigma_* \left (   \nabla_x G_s ( {\vec x}, {\vec z})   - {\vec D}^2_x G_s ({\vec x}, {\vec y})  ({\vec y}- {\vec z})
\right )
 \times \left ( {\vec F} ({\vec y} ) + \alpha {\vec w}_{0, fs} \left( \frac{y-z}{\alpha} \right ) \right ) \dif {\vec y} .  \nonumber
 \end{align}
 First considering ${\vec I}_{1,1} $, using (\ref{thm:wdiff2}) we estimate that
\begin{align}
 | {\vec I}_{1,1}  | \le & C  \sigma_*  \alpha^{3/2} \left  \| {\vec E}_\alpha - {\vec E}_0 - {\vec w}_{fs} -   \left ( \frac{\sigma_*-\sigma_s}{\sigma_*}\right ) \nabla \phi_0 \right \|_{L^2(B_\alpha)}  \nonumber \\
 \le & C   \sigma_*   \alpha^6 ( \mu_{r,*} |1-  \mu_{r,*}^{-1} | +\mu_{r,*} \nu ) \| \nabla \times {\vec E}_0 \|_{W^{2,\infty}(B_\alpha)} 
  \nonumber \\
 \le & C   \nu \alpha^4 \mu_{r,s}
 ( \mu_{r,*} | 1-  \mu_{r,*}^{-1} | +\mu_{r,*} \nu  ) \|  {\vec H}_0 \|_{W^{2,\infty}(B_\alpha)} 
   \nonumber  \\
 \le & C   \nu \alpha^4  \|  {\vec H}_0 \|_{W^{2,\infty}(B_\alpha)} , \nonumber  
 \end{align}
 Secondly, considering ${\vec I}_{1,2} $, using (\ref{eqn:e0mfmphiorg}), (\ref{eqn:e0mfmphi}) we estimate that
 \begin{align}
 | {\vec I}_{1,2}  | \le & C \alpha^{3/2}   \left ( \sigma_* \left  \| {\vec E}_0 + \nabla \phi_0 -{\vec F} ({\vec y} ) \right \|_{L^2(B_\alpha)}  + \sigma_s \|  \nabla \phi_0  \|_{L^2(B_\alpha)}  \right )
   \nonumber \\
 \le & C \alpha^{3/2}\left ( \sigma_* \alpha  \alpha^{7/2}\| \nabla \times {\vec E}_0 \|_{W^{2,\infty}(B_\alpha)} + \varepsilon  \frac{\alpha^{5/2}}{D^2} \|  \nabla \times {\vec H}_0 \|_{W^{1,\infty}(B_\alpha)} \right )\nonumber \\
 \le & C \left (  \nu \alpha^4 \mu_{r,s} \|  {\vec H}_0 \|_{W^{2,\infty}(B_\alpha)} + \nu  \frac{\alpha^{4} }{D^2} \| \nabla \times {\vec H}_0 \|_{W^{1,\infty}(B_\alpha)}  \right )  , \nonumber \\
 \le & C   \nu \alpha^4  \|  {\vec H}_0 \|_{W^{2,\infty}(B_\alpha)} \nonumber
 \end{align}
 Thirdly, considering  ${\vec I}_{1,3} $ using (\ref{eqn:alpham1F}) and (\ref{eqn:expandw0fs})  we have
  \begin{align}
 | {\vec I}_{1,3}  | \le & C  \sigma_* \alpha^{3/2} \alpha^2 \left \| {\vec F} ({\vec y} ) + \alpha {\vec w}_{0, fs} \left( \frac{y-z}{\alpha} \right ) \right \|_{L^2(B_\alpha)} \nonumber \\
 \le & C  \sigma_* \alpha^{7/2} \alpha^{5/2}  \omega \mu_0 \mu_{r,s}  \|{\vec H}_0 \|_{W^{2,\infty}(B_\alpha)} \nonumber \\
 \le & C \nu \alpha^4  \|{\vec H}_0 \|_{W^{2,\infty}(B_\alpha)} . \nonumber
 \end{align}
 Finally, following similar steps to~\cite{Ammari2014}
   leads to
 \begin{align}
 {\vec I}_{1,4} =  - \im \nu \alpha^3 \mu_{r,s}  \sum_{i=1}^3\int_{B} {\vec D}^2 G_s ({\vec x},{\vec z}){\vec \xi} \times ( {\vec e}_i  \times {\vec \xi} + {\vec \theta}_i ) \dif {\vec \xi} ({\vec H}_0({\vec z}))_i +{\vec R}({\vec x}), \nonumber
 \end{align}
  {with $|{\vec R}({\vec x}) | \le C \alpha^4   \|{\vec H}_0 \|_{W^{2,\infty}(B_\alpha)} $ and since { $| \mu_{r,s}-1 | \le  \frac{\alpha}{D}  $ then}
 \begin{align}
 {\vec I}_{1,4} =  - \im \nu \alpha^3   \sum_{i=1}^3\int_{B} {\vec D}^2 G_s ({\vec x},{\vec z}){\vec \xi} \times ( {\vec e}_i  \times {\vec \xi} + {\vec \theta}_i ) \dif {\vec \xi} ({\vec H}_0({\vec z}))_i +{\vec R}({\vec x}),
 \end{align}
 Furthermore,  since   $| {\vec I}_{1,1} +  {\vec I}_{1,2} + {\vec I}_{1,3}   | \le C \alpha^4  \|{\vec H}_0 \|_{W^{2,\infty}(B_\alpha)} $ they are also be grouped with ${\vec R}({\vec x})$.}

\noindent{\bf Consider term ${\vec I}_2$} \\
We write ${\vec I}_2$ as
\begin{align}
{\vec I}_2  = {\vec I}_{2,1} +  {\vec I}_{2,2} + {\vec I}_{2,3} + {\vec I}_{2,4}
, \nonumber 
\end{align} 
where
\begin{align}
{\vec I}_{2,1}  & := \int_{B_\alpha}  \sigma_s \nabla_x G_s ( {\vec x}, {\vec y}) \times \left ( {\vec E}_\alpha - {\vec E}_0 - {\vec w}_{fs} -\left ( \frac{\sigma_*-\sigma_s}{\sigma_*} \right )
\nabla \phi_0 \right ) \dif {\vec y} ,  \nonumber \\
{\vec I}_{2,2}  & := \int_{B_\alpha}  \sigma_s \nabla_x G_s ( {\vec x}, {\vec y}) \times \left ( {\vec E}_0 + \nabla \phi_0 -{\vec F} ({\vec y}) \right ) \dif {\vec y}  
-\int_{B_\alpha} \sigma_s\frac{\sigma_s}{\sigma_*} \nabla_x G_s ( {\vec x}, {\vec y})\times \nabla \phi_0 \dif {\vec y} ,
 \nonumber \\
{\vec I}_{2,3}  & := \int_{B_\alpha}  \sigma_s \left (  \nabla_x G_s ( {\vec x}, {\vec y}) - \nabla_x G_s ( {\vec x}, {\vec z})   + {\vec D}^2_x G_s ({\vec x}, {\vec y})  ({\vec y}- {\vec z})
\right )
 \times \left ( {\vec F} ({\vec y} ) + \alpha {\vec w}_{0, fs} \left( \frac{y-z}{\alpha} \right ) \right ) \dif {\vec y},  \nonumber \\
{\vec I}_{2,4}  & := \int_{B_\alpha}  \sigma_s \left (   \nabla_x G_s ( {\vec x}, {\vec z})   - {\vec D}^2_x G_s ({\vec x}, {\vec y})  ({\vec y}- {\vec z})
\right )
 \times \left ( {\vec F} ({\vec y} ) + \alpha {\vec w}_{0, fs} \left( \frac{y-z}{\alpha} \right ) \right ) \dif {\vec y} .  \nonumber
 \end{align}

First considering ${\vec I}_{2,1} $ in a similar way to ${\vec I}_{1,1} $, using  (\ref{thm:wdiff2})  we estimate that
 \begin{align}
 | {\vec I}_{2,1}  | \le & C  \sigma_s \alpha^{3/2} \left  \| {\vec E}_\alpha - {\vec E}_0 - {\vec w}_{fs} -  \left ( \frac{\sigma_*-\sigma_s}{\sigma_*} \right ) \nabla \phi_0 \right \|_{L^2(B_\alpha)}  \nonumber \\
 \le & C   \sigma_s \alpha^6  ( \mu_{r,*} | 1-  \mu_{r,*}^{-1} | + \nu \mu_{r,*}   ) \| \nabla \times {\vec E}_0 \|_{W^{2,\infty}(B_\alpha)}
   \nonumber \\
    \le & C  \varepsilon  \frac{\alpha^6}{D^2} \mu_{r,s} ( \mu_{r,*} | 1 -  \mu_{r,*}^{-1} | + \nu \mu_{r,*}   ) \|  {\vec H}_0 \|_{W^{2,\infty}(B_\alpha)} 
  \nonumber  \\
 \le & C   \nu \frac{\alpha^6}{D^2}   \|  {\vec H}_0 \|_{W^{2,\infty}(B_\alpha)} 
  \nonumber 
 \end{align}
Secondly, considering ${\vec I}_{2,2} $ in a similar way to ${\vec I}_{2,2} $ , using  (\ref{eqn:e0mfmphiorg}),  (\ref{eqn:e0mfmphi})
 we estimate that
  \begin{align}
 | {\vec I}_{2,3}  | \le & C \alpha^{3/2} \left (  \sigma_s \left  \| {\vec E}_0 + \nabla \phi -{\vec F} ({\vec y} \right \|_{L^2(B_\alpha)} +  \sigma_s\frac{\sigma_s}{\sigma_*} \| \nabla \phi_0 \|_{L^2(B_\alpha)}
 \right )
   \nonumber \\
 \le & C \alpha^{3/2} \left (   \sigma_s \alpha^{9/2} \| \nabla \times {\vec E}_0 \|_{W^{2,\infty}(B_\alpha)}+    \sigma_s \frac{\sigma_s}{\sigma_*}  \alpha^{5/2} \omega \mu_0 \| \nabla \times {\vec H}_0\|_{W^{1,\infty}(B_\alpha)} \right )\nonumber \\
 \le & C \left ( \varepsilon   \frac{ \alpha^6}{D^2} \mu_{r,s} \|  {\vec H}_0 \|_{W^{2,\infty}(B_\alpha)} + \varepsilon \frac{  \alpha^4}{D^2}\frac{\sigma_s}{\sigma_*} \| \nabla \times {\vec H}_0\|_{W^{1,\infty}(B_\alpha)}  \right )  \nonumber  \\
 \le & C \left (  \nu\frac{ \alpha^6}{D^2}  \|  {\vec H}_0 \|_{W^{2,\infty}(B_\alpha)} +\nu \frac{  \alpha^4}{D^2} \frac{\sigma_s}{\sigma_*} \| \nabla \times {\vec H}_0\|_{W^{1,\infty}(B_\alpha)}  \right ) . \nonumber  
 \end{align}
  Thirdly, considering  ${\vec I}_{2,3} $ in a similar way to  ${\vec I}_{1,3} $ using (\ref{eqn:alpham1F}) and (\ref{eqn:expandw0fs})  we have
  \begin{align}
 | {\vec I}_{2,3}  | \le & C \sigma_s \alpha^{3/2} \alpha^2 \left \| {\vec F} ({\vec y} ) + \alpha {\vec w}_{0, fs} \left( \frac{y-z}{\alpha} \right ) \right \|_{L^2(B_\alpha)} \nonumber \\
 \le & C  \sigma_s \alpha^{7/2} \alpha^{5/2}  \omega \mu_0 \mu_{r,s}  \|{\vec H}_0 \|_{W^{2,\infty}(B_\alpha)} \nonumber \\
 \le & C  \varepsilon  \frac{\alpha^{6}}{D^2} \|{\vec H}_0 \|_{W^{2,\infty}(B_\alpha)}  \nonumber\\
 \le & C  \nu \frac{\alpha^{6}}{D^2}  \|{\vec H}_0 \|_{W^{2,\infty}(B_\alpha)} . \nonumber
 \end{align}
 Fourthly, considering   ${\vec I}_{2,4} $
 \begin{align}
| {\vec I}_{2,4} | & =\left |  \int_{B_\alpha}  \sigma_s \left (    {\vec D}^2_x G_s ({\vec x}, {\vec y})  ({\vec y}- {\vec z})
\right )
 \times \left ( {\vec F} ({\vec y} ) + \alpha {\vec w}_{0, fs} \left( \frac{y-z}{\alpha} \right ) \right ) \dif {\vec y} \right |  \nonumber\\
 & \le C \mu_{r,s} \sigma_s \omega \mu_0 \alpha^5 \| {\vec H}_0 \|_{W^{2,\infty } (B_\alpha) }\nonumber \\
 & \le C  \nu \frac{\alpha^{5}}{D^2} \| {\vec H}_0 \|_{W^{2,\infty } (B_\alpha) }, \nonumber 
 \end{align}
where the term involving $\nabla_x G({\vec x},{\vec z})$ cancels by  following similar steps to~\cite{Ammari2014}. Thus ${\vec I}_2$ can be absorbed into ${\vec R}({\vec x})$.

\noindent {\bf Consider term ${\vec I}_3$}\\
Noting that ${\vec H}_\alpha = \frac{1}{\im \omega \mu_0 \mu_{r,*}} \nabla \times {\vec E}_\alpha$ and ${\vec H}_0 = \frac{1}{\im \omega \mu_0 \mu_{r,s}}  \nabla \times {\vec E}_0$ in $B_\alpha$, we introduce
\begin{align}
  {\vec H}_{0,fs}^* ({\vec \xi}) := \frac{1}{\im \omega \mu_0 \mu_{r,s} } \nabla_\xi \times {\vec w}_{0,fs}^* . \nonumber
\end{align} 
We write ${\vec I}_3$  as
\begin{align}
{\vec I}_3  = -   \left ( {\vec I}_{3,1} +  {\vec I}_{3,2} + {\vec I}_{3,3} + {\vec I}_{3,4} \right ) , \nonumber 
\end{align}
where
\begin{align}
{\vec I}_{3,1}:=& (\mu_{r,s}-\mu_{r,*} )  \int_{B_\alpha} {\vec D}_x^2 G_s ( {\vec x}, {\vec y}) \left ( {\vec H}_\alpha ({\vec y})  - \frac{\mu_{r,s}}{\mu_{r,*} } {\vec H}_{0}({\vec y} )-  \frac{\mu_{r,s}}{\mu_{r,*} } {\vec H}_{0,fs}^* \left (\frac{{\vec y}- {\vec z}}{\alpha} \right ) \right ) \dif {\vec y}  \nonumber , \\
{\vec I}_{3,2}:=&  \left (\frac{\mu_{r,s}}{ \mu_{r,*} } - 1 \right ) \mu_{r,s}  \int_{B_\alpha} (  {\vec D}_x^2 G_s ( {\vec x}, {\vec y}) - {\vec D}_x^2 G_s ( {\vec x}, {\vec z}) )
\left (   {\vec H}_0({\vec y}) + {\vec H}_{0,fs}^* \left (\frac{{\vec y}- {\vec z}}{\alpha} \right ) \right )  \dif {\vec y} , \nonumber\\
{\vec I}_{3,3}:=&  \left (\frac{\mu_{r,s}}{ \mu_{r,*} } - 1 \right ) \mu_{r,s}   \int_{B_\alpha}   {\vec D}_x^2 G_s ( {\vec x}, {\vec z}) ({\vec H}_0({\vec y}) - {\vec H}_0({\vec z}) ) \dif {\vec y} ,  \nonumber\\
{\vec I}_{3,4}:=& \left (\frac{\mu_{r,s}}{ \mu_{r,*} } - 1 \right )  \mu_{r,s} \int_{B_\alpha}   {\vec D}_x^2 G_s ( {\vec x}, {\vec z}) \left ({\vec H}_0({\vec z}) + {\vec H}_{0,fs}^*\left  (\frac{{\vec y} - {\vec z}}{\alpha} \right)\right ) \dif {\vec y} . \nonumber
\end{align}
Then, for ${\vec I}_{3,1}$, using  (\ref{lemma_engy1}) we have
\begin{align}
| {\vec I}_{3,1} | \le & C  \alpha^{3/2}|\mu_{r,s}-\mu_{r,*}| \left  \|  {\vec H}_\alpha ({\vec y})  - \frac{\mu_{r,s}}{\mu_{r,*} } {\vec H}_0({\vec y} )-  \frac{\mu_{r,s}}{\mu_{r,*} } {\vec H}^*({\vec y}) \right \|_{L^2(B_\alpha)}\nonumber\\
\le & C \alpha^{3/2}   \frac{1}{\omega \mu_0 \mu_{r,*}}  \left \| \nabla \times {\vec E}_\alpha - \nabla \times {\vec E}_0 - \nabla \times {\vec w} \  \right \|_{L^2(B_\alpha)}\nonumber\\
\le & C \alpha^{5} \mu_{r,s}   \left (  | 1 - \mu_{r,*}^{-1} |  + \nu  \right )  \|{\vec H}_0 \|_{W^{2,\infty}(B_\alpha)}   \nonumber\\
\le & C \alpha^{5} \nu \|{\vec H}_0 \|_{W^{2,\infty}(B_\alpha)} ,  \nonumber 
\end{align}

For ${\vec I}_{3,2}$
\begin{align}
|{\vec I}_{3,2}| \le &
C \alpha^{3/2}  \left | \frac{\mu_{r,s}}{\mu_{r,*}}-1 \right | \mu_{r,s} \alpha    \alpha^{3/2} \|  {\vec H}_0 \|_{W^{2,\infty}(B_\alpha)}  \nonumber \\
\le & C\alpha^{4} \|  {\vec H}_0 \|_{W^{2,\infty}(B_\alpha)} \nonumber ,
\end{align}
and for ${\vec I}_{3,3}$ we get
\begin{align}
|{\vec I}_{3,3}| \le &C  \alpha^{3/2} \left | \frac{\mu_{r,s}}{\mu_{r,*}}-1 \right | \mu_{r,s} \alpha \alpha^{3/2}  \|  {\vec H}_0 \|_{W^{2,\infty}(B_\alpha)}  \nonumber \\
\le & C\alpha^{4} \|  {\vec H}_0 \|_{W^{2,\infty}(B_\alpha)} \nonumber .
\end{align}
The fourth term ${\vec I}_{3,4}$ can be dealt with in a similar way to~\cite{Ammari2014}, where, in our case,
\begin{align}
{\vec H}_{0,fs}^* ({\vec \xi}) = \frac{1}{\im \omega \mu_0 \mu_{r,s} } \nabla_\xi \times {\vec w}_{0,fs}^* = 
\sum_{i=1}^3 \frac{1}{2} ({\vec H}({\vec z}))_i \nabla_\xi \times {\vec \theta}_i +\sum_{i,j=1}^3 \frac{\alpha}{3} ( {\vec D}{\vec H}({\vec z}))_{ij} \nabla_\xi \times {\vec \psi}_{ij}({\vec \xi})  ,
\end{align} 
follows from (\ref{eqn:expandw0fs})
and leads to
\begin{align}
{\vec I}_{3,4} = \left (\frac{\mu_{r,s}}{ \mu_{r,*} } - 1 \right )  \alpha^3 \mu_{r,s}  \sum_{i=1}^3 \int_B {\vec D}^2 G_s(  {\vec x}, {\vec z}) \left ( {\vec e}_i + \frac{1}{2} \nabla \times {\vec \theta}_i  \right ) \dif {\vec \xi}  ({\vec H}_0({\vec z}))_i  + {\vec R}({\vec x}) . \nonumber
\end{align}
 {Furthermore, since $\left | \left ( \frac{\mu_{r,s}}{\mu_{r,*}}-1 \right )  - \left ( \frac{1}{\mu_{r,*}}-1 \right )\right | = \left |
   \frac{\mu_{r,s}}{\mu_{r,*}} - \frac{1}{\mu_{r,s} }\right | \le C $ and  $|\mu_{r,s} - 1 | \le \frac{\alpha}{D}$,
\begin{align}
{\vec I}_{3,4} = \left (\frac{1}{ \mu_{r,*} } - 1 \right )  \alpha^3   \sum_{i=1}^3 \int_B {\vec D}^2 G_s(  {\vec x}, {\vec z}) \left ( {\vec e}_i + \frac{1}{2} \nabla \times {\vec \theta}_i  \right ) \dif {\vec \xi}  ({\vec H}_0({\vec z}))_i  + {\vec R}({\vec x}) . \nonumber
\end{align}}
We note that $| {\vec I}_{3,1} + {\vec I}_{3,2} + {\vec I}_{3,3}  | \le C \alpha^4 \| {\vec H}_0 \|_{W^{2,\infty}(B_\alpha)}$ and so can be grouped with ${\vec R}({\vec x})$. Summing ${\vec I}_1$, ${\vec I}_2$ and ${\vec I}_3$ gives the desired result.
\end{proof}

\begin{remark} \label{remak:laplacegreen}
For sufficiently small $k_0$ we can replace $G_s({\vec x},{\vec z})$ by the free space Laplace  Green's function $G_0({\vec x},{\vec z})$ leading to
\begin{align}
({\vec H}_\alpha - {\vec H}_0)({\vec x}) = & - \im \nu \alpha^3  \sum_{i=1}^3\int_{B} {\vec D}_x^2 G_0 ({\vec x},{\vec z}){\vec \xi} \times ( {\vec e}_i  \times {\vec \xi} + {\vec \theta}_i ) \dif {\vec \xi} ({\vec H}_0({\vec z}))_i 
\nonumber \\
& + \left (1- \frac{1}{ \mu_{r,*} }  \right )  \alpha^3 \sum_{i=1}^3 \int_B {\vec D}_x^2 G_0 (  {\vec x}, {\vec z}) \left ( {\vec e}_i + \frac{1}{2} \nabla \times {\vec \theta}_i  \right ) \dif {\vec \xi} ({\vec H}_0({\vec z}))_i  + {\vec R}({\vec x}) .
\end{align}
\end{remark}

\section{Tensor representation} \label{sect:tensor}

Applying  results from~\cite{LedgerLionheart2015}, ~\cite{LedgerLionheart2020spect} and~\cite{Ammari2015} allows us to obtain the alternative form of Theorem~\ref{eqn:mainresult}.
\begin{theorem}
 {Let $\nu$ be order one and let $\alpha$ be small with $\epsilon \le  \nu$.}
For ${\vec x} \in \Omega_{fs}$ {and objects away from $\Gamma_s$}
\begin{align}
({\vec H}_\alpha - {\vec H}_0 ) ({\vec x})_i = ({\vec D}_x^2 G_s ({\vec x},{\vec z}))_{ij} ({\mathcal M})_{jk} ({\vec H}_0({\vec z})_k + ({\vec R}({\vec x}))_i ,  \label{eqn:tenrep}
\end{align}
where $|{\vec R}({\vec x})| \le C\alpha^4 \| {\vec H}_0 ({\vec z})\|_{W^{2,\infty} (B_\alpha)}$ uniformly in ${\vec x}$ in a compact set away from $B_\alpha$ and $\mathcal{M}$ is a complex symmetric rank 2 magnetic polarizability tensor, whose coefficients $({\mathcal M})_{ij}$ are independent of the object's position and $({\mathcal M})_{ij}:=(\tilde{\mathcal{R}})_{ij}+\im(\mathcal{I})_{ij}=(\mathcal{N}^0)_{ij}+(\mathcal{R})_{ij}+\im(  \mathcal{I})_{ij}$ with
\begin{subequations}
\label{eqn:NRI}
\begin{align}
(\mathcal{N}^0( \alpha B,\mu_{r,*}) )_{ij}&:=\alpha^3\delta_{ij}\int_{B}(1-\tilde{\mu}_r^{-1})\dif \bm{\xi}+\frac{\alpha^3}{4}\int_{B\cup B^c}\tilde{\mu}_r^{-1}\nabla\times\tilde{\bm{\theta}}_i^{(0)}\cdot\nabla\times\tilde{\bm{\theta}}_j^{(0)}\dif \bm{\xi},\\
(\mathcal{R}(\alpha B, \omega,\sigma_*,\mu_{r,*}))_{ij}&:=-\frac{\alpha^3}{4}\int_{B\cup B^c}\tilde{\mu}_r^{-1}\nabla\times\overline{\bm{\theta}_i^{(1)}}\cdot\nabla\times{\bm{\theta}_j^{(1)}}\dif \bm{\xi}, \label{eqn:Rtensor}\\
(\mathcal{I}(\alpha B, \omega,\sigma_*,\mu_{r,*}))_{ij}&:=\frac{\alpha^3}{4}\int_B\nu\Big(\overline{\bm{\theta}_i^{(1)}+\tilde{\bm{\theta}}_i^{(0)}+\bm{e}_i\times\bm{\xi}}\Big)\cdot\Big({\bm{\theta}_j^{(1)}+\tilde{\bm{\theta}}_j^{(0)}+\bm{e}_j\times\bm{\xi}}\Big)\dif \bm{\xi}, \label{eqn:Itensor}
\end{align}
\end{subequations}
are each the coefficients of real symmetric rank-2 tensors and $\delta_{ij}$ is the Kronecker delta. In the above,  the overbar denotes the complex conjugate,  $\tilde{\mu}_r ({\vec \xi}) =\mu_{r,*}$ in $B$ and $\tilde{\mu}_r({\vec \xi}) =1$ otherwise, where ${\vec \xi}$ is measured from the origin, which lies inside $B$. Additionally, ${\vec \theta}_i = {\vec \theta}_i^{(0)} + {\vec \theta}_i^{(1)} - {\vec e}_i \times {\vec \xi}= \tilde{\vec \theta}_i^{(0)} + {\vec \theta}_i^{(1)} $ is a splitting of the solution to (\ref{eqn:thetatrans}).
\end{theorem}
\begin{proof}
The reduction of (\ref{eqn:mainresult}) to an asymptotic formula in terms of a  complex symmetric rank 2 magnetic polarizability tensor follows the same steps as the proof of Theorem 3.2  in~\cite{LedgerLionheart2015}. Theorem 5.1 in~\cite{LedgerLionheart2020spect} establishes the  form $({\mathcal M})_{ij}:=(\tilde{\mathcal{R}})_{ij}+\im(\mathcal{I})_{ij}=(\mathcal{N}^0)_{ij}+(\mathcal{R})_{ij}+\im(  \mathcal{I})_{ij}$  and explicit expressions for the tensor coefficients in (\ref{eqn:NRI}). The independence of object position of $({\mathcal M})_{ij}$ follows from Proposition 5.1 in~\cite{Ammari2015}.
\end{proof}

\begin{remark}
Considering a small measurement coil with support in $\Omega_{fs}$, then the induced voltage in the coil has the form
\begin{align}
\Delta V = \im \omega \mu_0 \int_S {\vec n} \cdot ({\vec H}_\alpha - {\vec H}_0 ) ({\vec x}) \dif {\vec x},
\nonumber
\end{align}
where ${\vec n}$ is perpendicular to the open surface $S$ defining the plane of the coil. Substituting (\ref{eqn:tenrep}) leads to
\begin{align}
\Delta V_s = ({\vec H}_0^{ms}({\vec z}))_i ({\mathcal M})_{ij}  ({\vec H}_0({\vec z}))_j + R, \qquad  {\vec H}_0^{ms}({\vec z})_i: =  \im \omega \mu_0 \int_S ({\vec D}_x^2 G_s ({\vec x},{\vec z}))_{ij} ({\vec n})_j \dif {\vec x},
\label{eqn:vsincludingsoilgs}
\end{align}
where ${\vec H}_0^{ms}({\vec z})$ is the magnetic field that would be produced by the measurement coil evaluated at ${\vec z}$ in absence of the object, but in the presence of the soil, if the measurement coil acts as an exciter and $ |R| \le C \alpha^4 \|{\vec H}_0 \|_{W^{2,\infty}(B_\alpha)}$. Furthermore, for sufficiently small $k_0$, following Remark~\ref{remak:laplacegreen}, we also have
\begin{align}
\Delta V_{s0} = ({\vec H}_0^{ms}({\vec z}))_i ({\mathcal M})_{ij}  ({\vec H}_0({\vec z}))_j + R, \qquad  {\vec H}_0^{ms}({\vec z})_i: =  \im \omega \mu_0 \int_S ({\vec D}_x^2 G ({\vec x},{\vec z}))_{ij} ({\vec n})_j \dif {\vec x}, 
\end{align}
with ${\vec H}_0^{ms}({\vec z})$  being as before, but in the absence of the soil, with differences being absorbed in to $R$.
\end{remark}

\section{Numerical simulations} \label{sect:numerical}

\subsection{Buried spherical object} \label{results:sphere}

We fix $\Gamma_s$ to be the plane $x_3=0$ so that  $x_3 >0$ is $\Omega_{fs}$ and  $x_3 <0 $ is $\Omega_s$. Unless otherwise mentioned, for this section $B_\alpha \subset \Omega_s $ is a sphere of radius $\alpha=0.1$ m  with centre at ${\vec z}=(0,0,-0.4)$ m and material properties $\sigma_* = 1 \times 10^6$ S/m,  $\mu_{r,*}=1$. The background field ${\vec H}_0$ is generated by a  cylindrical current source with ${\vec J}^s = (10 /  A_{\text{coil}}) {\vec e}_\phi$ $\text{A/m}^2$, with inner radius $r_\text{{inner}}=0.12$ m, outer radius $r_{\text{outer}}=0.15$ m,  height $h=0.1$ m so that $A_{\text{coil}} =h (r_{\text{outer}}-r_{\text{inner}})= 3\times10^{-3}$ m$^2$ and located $h_{off}=0.2$ m above $\Gamma_s$ in $\Omega_{fs}$. An identically sized cylinder is used to measure the induced voltage. This situation is illustrated in Figure~\ref{fig:sim1mesh}.

\begin{figure}[h]
\begin{center}
$\begin{array}{cc}
\includegraphics[width=3in]{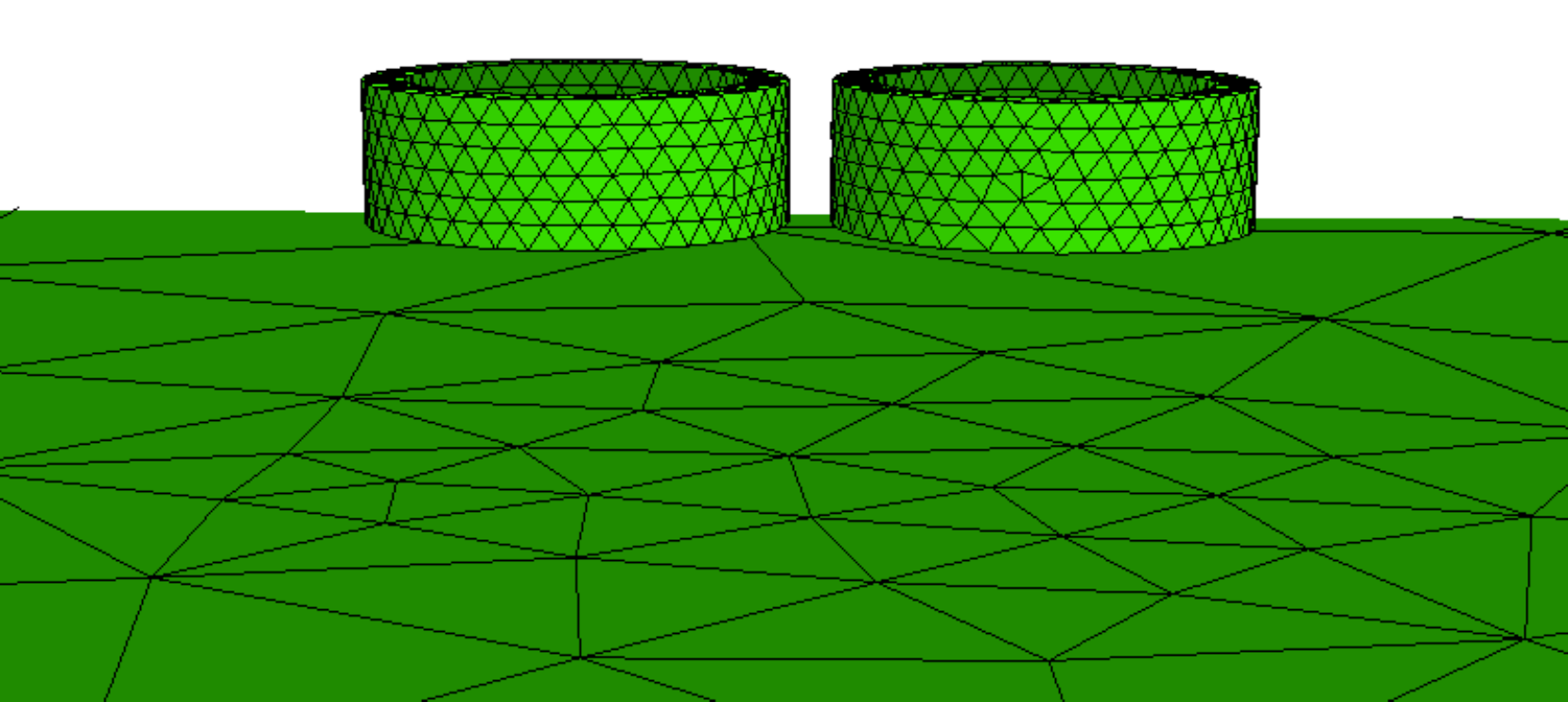} &
\includegraphics[width=3in]{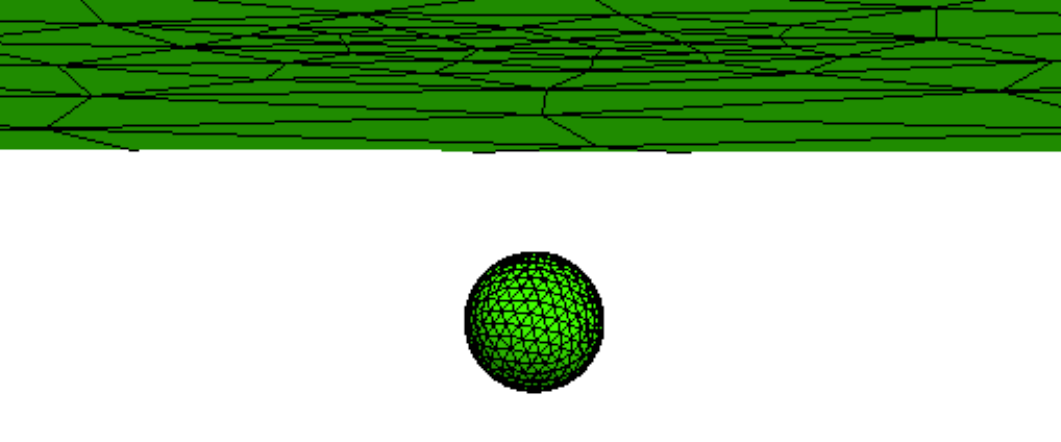} \\
(a) & (b)
\end{array}$
\end{center}
\caption{Buried spherical object: Discretised by 22,834 unstructured tetrahedra and prismatic elements showing  the surface triangulations on $(a)$ the half plane and excitation and measurement coils and $(b)$ the half plane and spherical object (prismatic layers are used to model the thin skin depths inside this object). } \label{fig:sim1mesh}
\end{figure}

The unbounded $\Omega_{fs}$ and $\Omega_s$ are truncated a finite distance from the sphere to form a computational domain  $\Omega$ in the form of a  box with side lengths of $20$ m. The computational domain is discretised by  22, 834 unstructured tetrahedra and prismatic elements using  the Netgen mesh generator~\cite{netgendet}, which also forms also  a discretisation of the coils and  object. To compute numerical solutions, we apply a vector potential formulation of the eddy current problem, where a vector potential ${\vec A}_\alpha$ is introduced such that ${\vec H}_\alpha = \mu^{-1} \nabla \times {\vec A}_\alpha$ and use the boundary condition ${\vec n} \times {\vec A}_\alpha= {\vec 0}$ as an approximation to static decay on $\partial \Omega$. The scheme  is regularised to circumvent the Coulomb gauge in the non-conducting by the introduction of a small regularisation parameter $\varepsilon_{\text{reg}}= 1 \times 10^{-6} $ and using a high-order ${\vec H}(\hbox{curl})$ conforming finite elements, with those basis functions corresponding to gradient fields skipped in this region~\cite{ledgerzaglmayr2010}. Specifically, the basis functions proposed by Sch\"oberl and Zaglmayr~\cite{SchoberlZaglmayr2005} are employed and the NG-Solve finite element library employed for the computation. On the generated mesh, the element order is uniformly increased until solution convergence is achieved, which, for the frequencies considered occurs with $p=3$ elements. 

In absence of the object (by fixing its material parameters to be either those of free space or soil), the resulting solution for induced voltage in the measurement coil as a function of $\omega$ due to the presence of the soil is
\begin{align}
\Delta V_0 = \im \omega \mu_0 \int_{S} ( {\vec H}_{0}- {\vec H}_{0,fs})({\vec x})_i ({\vec n})_i   \dif {\vec x},
\end{align}
where ${\vec H}_{0}$ is the background magnetic field in the presence of the soil and  ${\vec H}_{0,fs}$ is the background magnetic field assuming the soil is the same as free space. Figure~\ref{fig:voltsoilonlyvarysigma} shows the results obtained for the situation where $\mu_{r,s}=1$ and $\sigma_s=0.01,0.1,1$ S/m in turn, where each line indicates a different soil conductivity. This figure illustrates the increasing influence  of the soil on the background field as the frequency increases, and also as the conductivity of the soil increases, illustrating the importance of considering the soil's conductivity in the model.
  \begin{figure}
 \begin{center}
 $\begin{array}{cc}
 \includegraphics[width=3in]{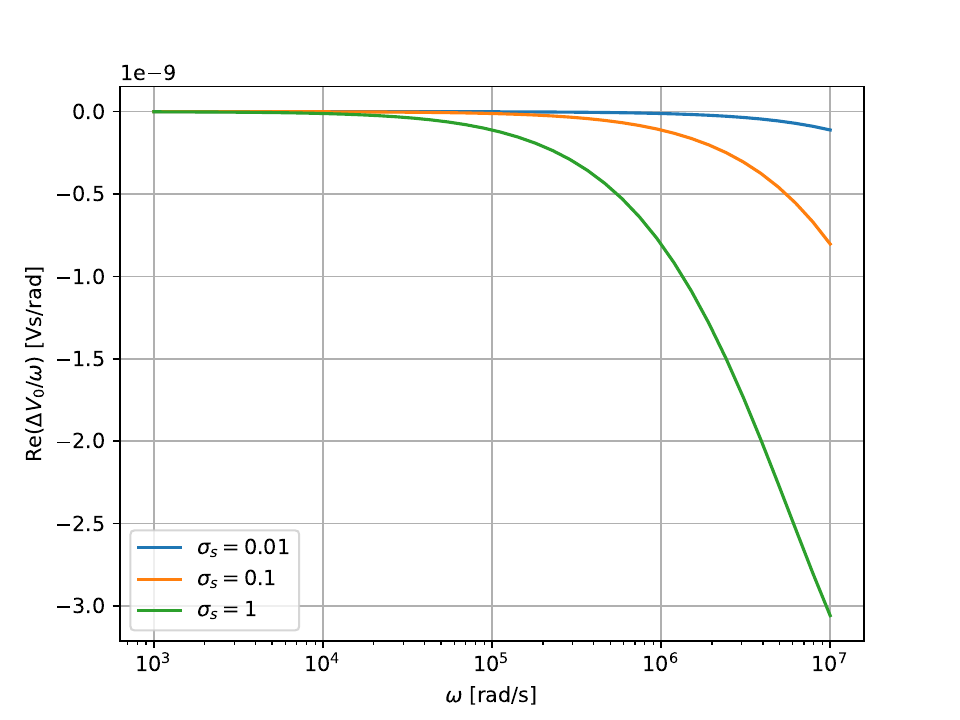} &
\includegraphics[width=3in]{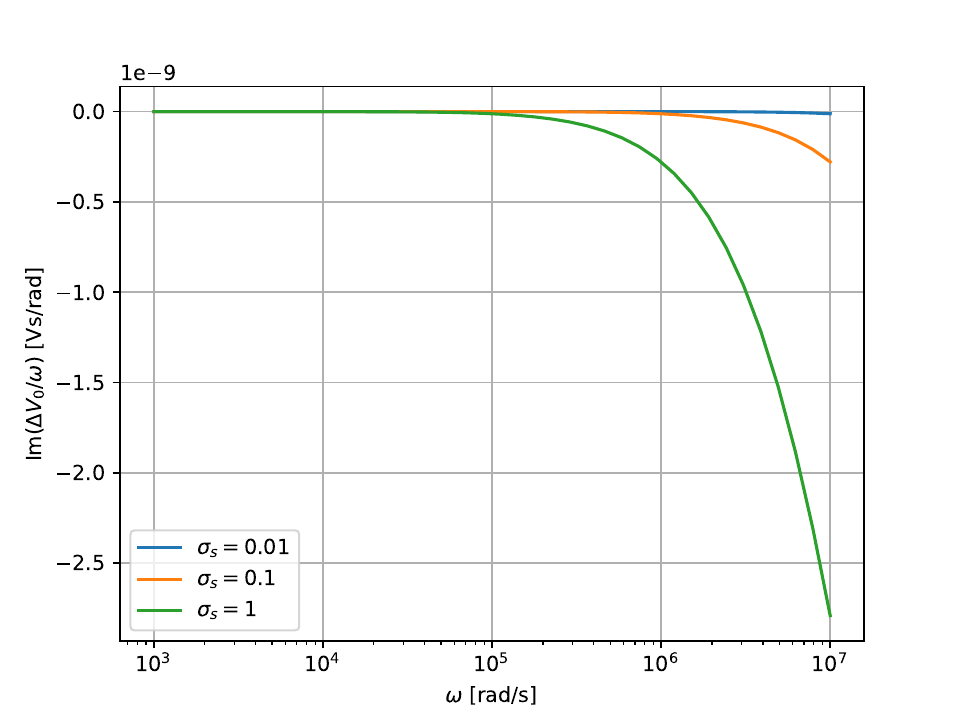} \\
(a) & (b) 
 \end{array}$
 \end{center}
 \caption{ {Influence of the soil's conductivity $\sigma_s$ on the induced voltage $\Delta V_0$ in the measurement coil as a function of $\omega$. Showing $(a)$ $\text{Re}( \Delta V_0 /\omega)$ and $(b)$ $\text{Im}( \Delta V_0 /\omega)$.}} \label{fig:voltsoilonlyvarysigma}
 \end{figure}
  Next, we consider the influence of the soil's magnetic permeability. We consider the extreme case of a soil with a high salinity, so that $\sigma_s=1.6$ S/m, and further consider  $\mu_{r,s}=1.0006,1.021,1.076$, the former representing typical soil conditions and latter two granite and iron rich soils, respectively. Figure~\ref{fig:voltsoilonlyvarymur} shows the results of this investigation where, compared to varying the soil's conductivity, for the presented values of $\mu_{r,s}$, their influence on the voltage perturbation $V_0$ is much smaller. Of course, if larger values of $\mu_{r,s}$ were considered, a stronger influence would be observed, but such values are not typically found in common soil types.
 
   \begin{figure}
 \begin{center}
 $\begin{array}{cc}
 \includegraphics[width=3in]{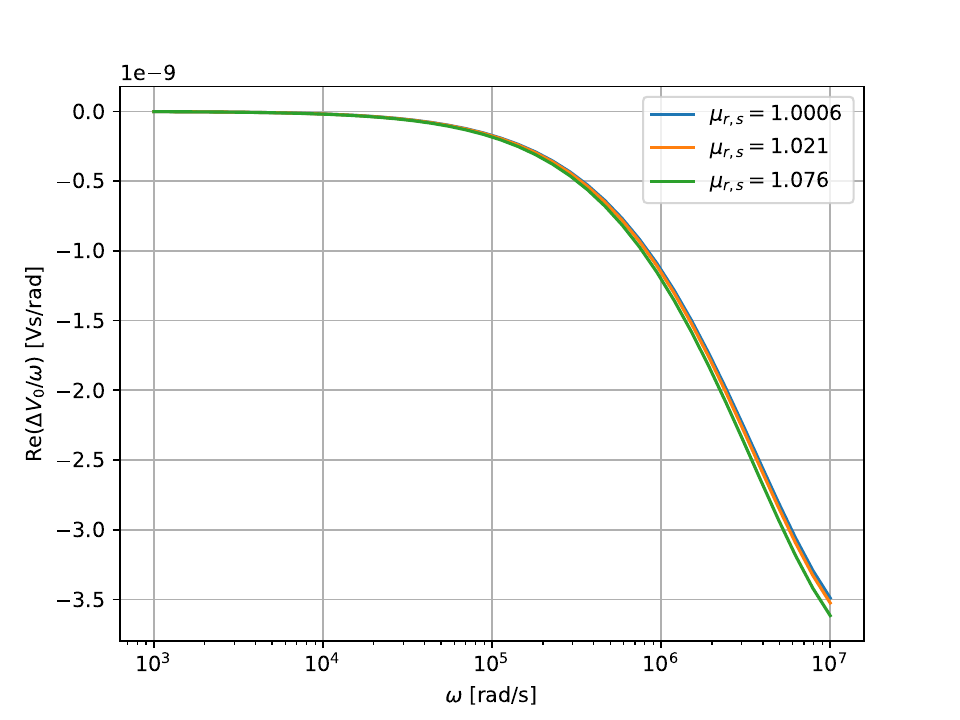} &
\includegraphics[width=3in]{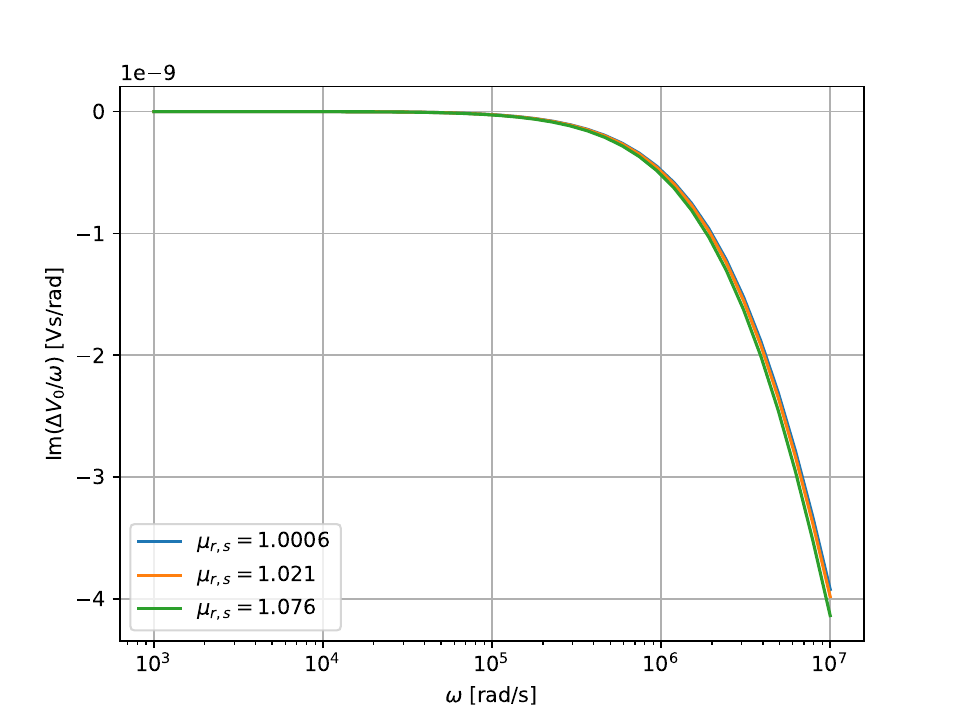} \\
(a) & (b) 
 \end{array}$
 \end{center}
 \caption{ {Influence of the soil's relative magnetic permeability $\mu_{r,s}$ on the induced voltage $\Delta V_0$ in the measurement coil as a function of $\omega$. Showing $(a)$ $\text{Re}( \Delta V_0 /\omega)$ and $(b)$ $\text{Im}( \Delta V_0 /\omega)$.}} \label{fig:voltsoilonlyvarymur}
 \end{figure}

We now turn attention to the influence of the soil on the induced voltage when the object is present. Considering soil with $\sigma_s =1.6$ S/m,  $\mu_{r,s}=1.0006$, we compare
\begin{align}
\Delta V_s = \im \omega \mu_0 \int_{S}  ( {\vec H}_{\alpha}- {\vec H}_{0,s})({\vec x})_i ({\vec n})_i   \dif {\vec x},  \qquad
\Delta V_{fs} = \im \omega \mu_0 \int_{S} ( {\vec H}_{\alpha}- {\vec H}_{0,fs})({\vec x})_i ({\vec n} )_i \dif {\vec x} ,
\end{align}
as a function of $\omega$ where $\Delta V_s$ is induced voltage due to the presence of the object compared to that of the soil and  $\Delta V_{fs} $ is the induced voltage due to the presence of the object compared to that of free space. Results for a range of different approaches are shown in Figure~\ref{fig:voltcomparemethsphere}. Firstly, we observe the significant difference between $\Delta V_{fs}$ Asym. (obtained by using the asymptotic expansion~\cite{LedgerLionheart2015,Ammari2014}, ${\vec H}_{0,fs}({\vec z})$ and the exact MPT for a sphere placed in free space~\cite{Wait1951}) and $\Delta V_s$ obtained by the complete finite element simulation, with significant differences for $\omega \ge 2 \times 10^5$ rad/s highlighting the importance of including the soil in the mathematical model. The curve $\Delta V_s$ Asym. is obtained by the new asymptotic expansion (\ref{eqn:tenrep}) 
and using
(\ref{eqn:vsincludingsoilgs}), and provides a significant improvement over  $\Delta V_{fs}$. A good approximation is obtained by $\Delta V_s$ Asym Aprox. (which uses a dipole model of both coils). As stated in Remark~\ref{remak:laplacegreen}, for sufficiently small $\omega$, we expect the Laplace Green's function $G_0$ to provide a reasonable approximation to $G_s$ and this is highlighted by the curves $\Delta V_{s0}$ Asym. (which uses $G_0$) and
 $\Delta V_{s0}$ Asym. Aprox. (which uses a dipole model of the coil and $G_0$). The comparison has been repeated for the sphere located at ${\vec z}= (0,0,-0.2),(0,0,-0.3), (0,0,-0.4)$ and $(0,0,-0.5)$ m and the agreement is similar to that shown in Figure~\ref{fig:voltcomparemethsphere}.

   \begin{figure}
 \begin{center}
 $\begin{array}{cc}
 \includegraphics[width=3in]{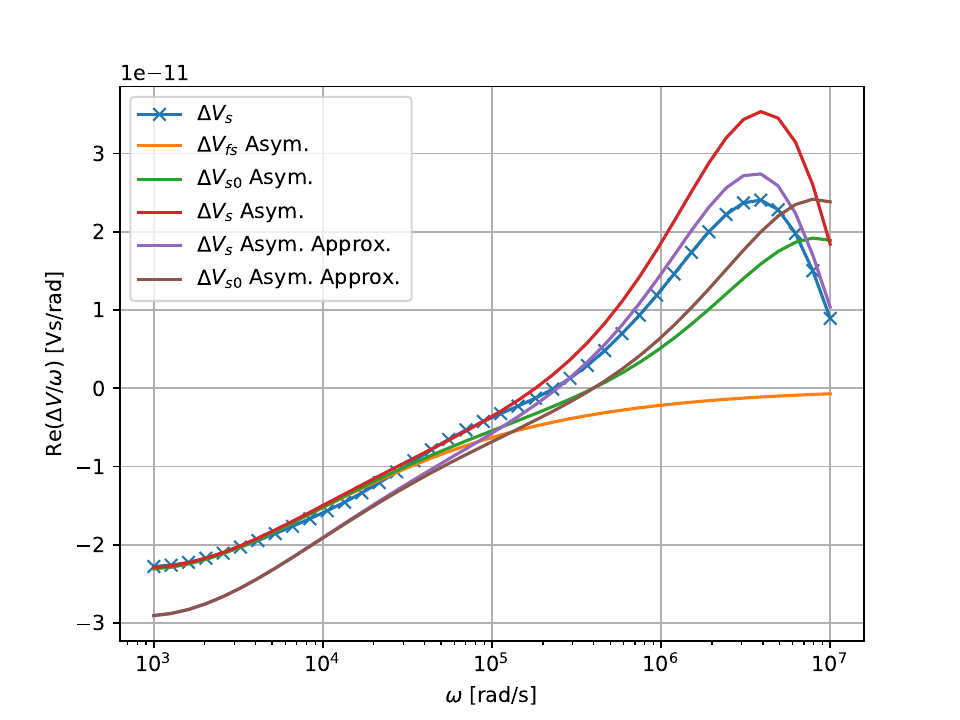} &
 \includegraphics[width=3in]{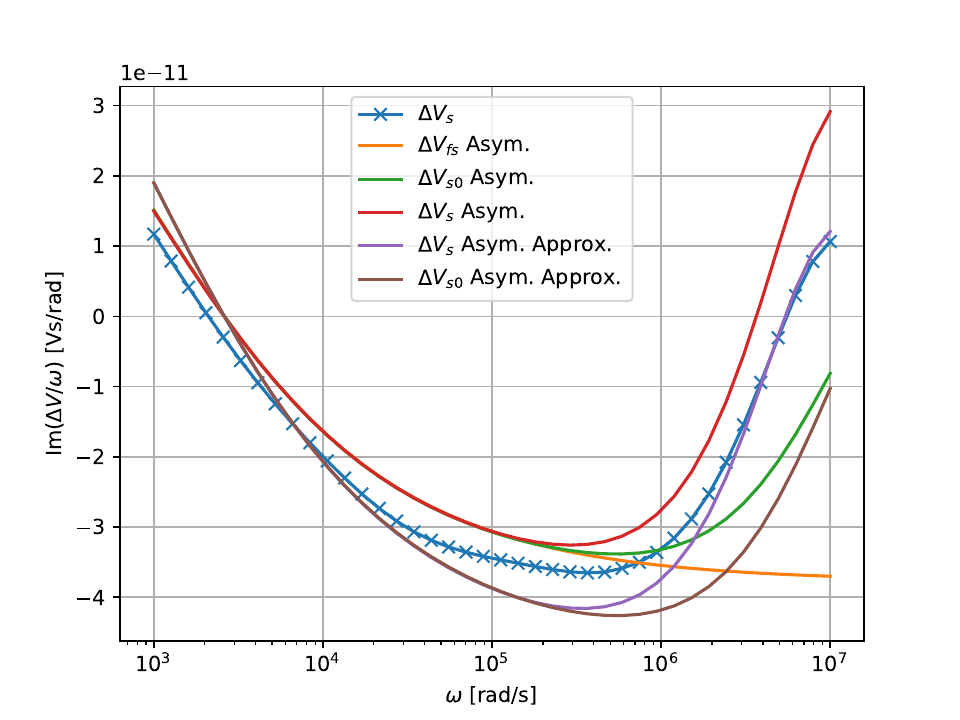}  \\
(a) & (b) 
 \end{array}$
 \end{center}
 \caption{ {Conducting buried sphere: Comparison of induced voltages $\Delta V_s$ and  $\Delta V_{fs}$ and  $\Delta V_{s0}$ in the measurement coil as a function of $\omega$. Showing $(a)$ $\text{Re}( \Delta V /\omega)$ and $(b)$ $\text{Im}( \Delta V /\omega)$.}} \label{fig:voltcomparemethsphere}
 \end{figure}
 
 For spheres located at ${\vec z}= (0,0,-0.2),(0,0,-0.3), (0,0,-0.4)$ and $(0,0,-0.5)$ m, in turn, the results obtained using $\Delta V_s$ Asym Aprox. as a function of $\omega$ are included in Figure~\ref{fig:voltcomparelocsphere}. As expected, objects buried at greater depths  produce weaker signals. For each object location a different mesh was generated, but each with approximately the same number of elements, and, as previously, using order $p=3$ elements ensured mesh convergence.
 \begin{figure}
 \begin{center}
 $\begin{array}{cc}
\includegraphics[width=3in]{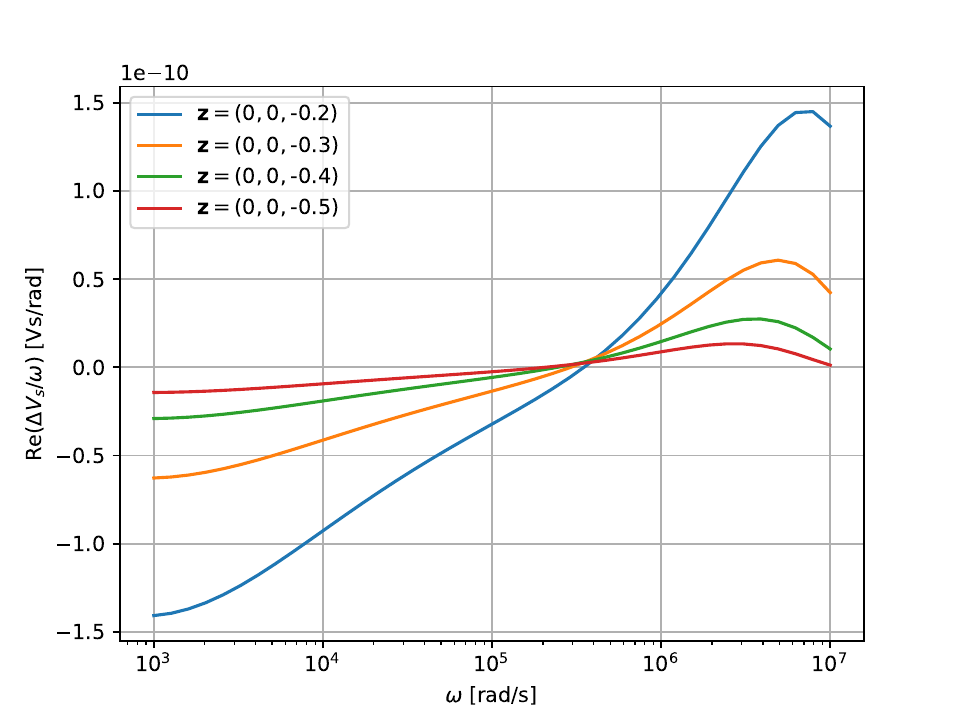} &
 \includegraphics[width=3in]{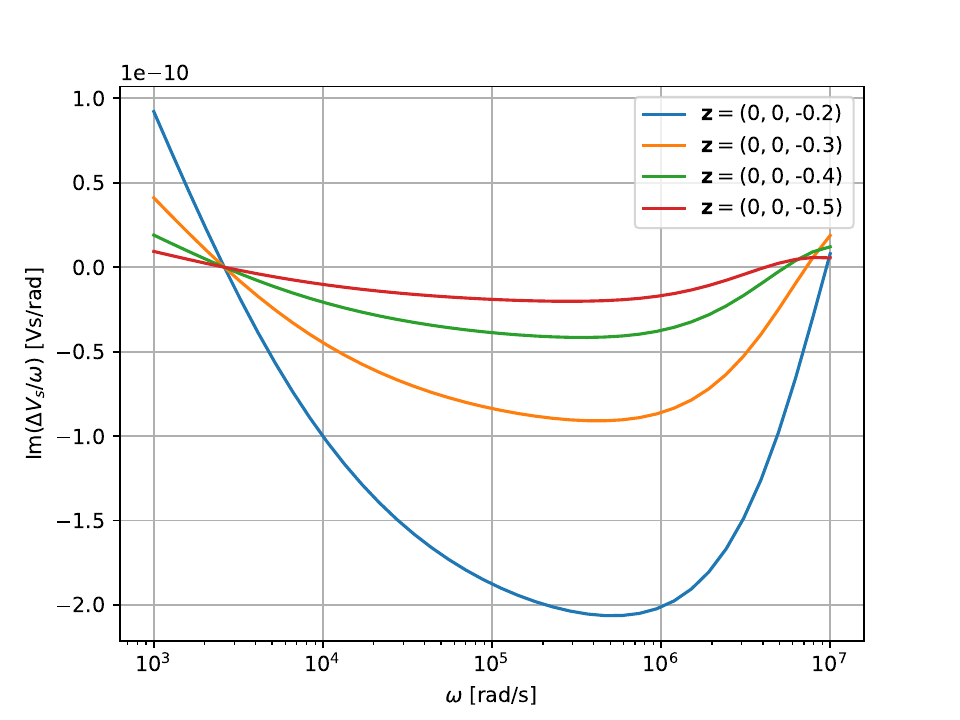} \\
 (a) & (b)
\end{array}$
 \end{center}
 \caption{ {Conducting buried sphere:  Influence of the object's location on $\Delta V_s$ in the measurement coil as a function of $\omega$. Showing $(a)$ $\text{Re}( \Delta V_s /\omega)$ and $(b)$ $\text{Im}( \Delta V_s /\omega)$.}} \label{fig:voltcomparelocsphere}
 \end{figure}
Again considering $\Delta V_s$ Asym Aprox. as a function of $\omega$,  returning the object location to ${\vec z}=(0,0,-0.4)$ m, but now considering the sphere's radius to be $\alpha=0.1,0.05,0.025, 0.0125$ m, in turn,  leads to the results shown  in Figure~\ref{fig:voltcomparealphasphere}. As expected, smaller objects lead to weaker voltage signals. 

 \begin{figure}
 \begin{center}
 $\begin{array}{cc}
\includegraphics[width=3in]{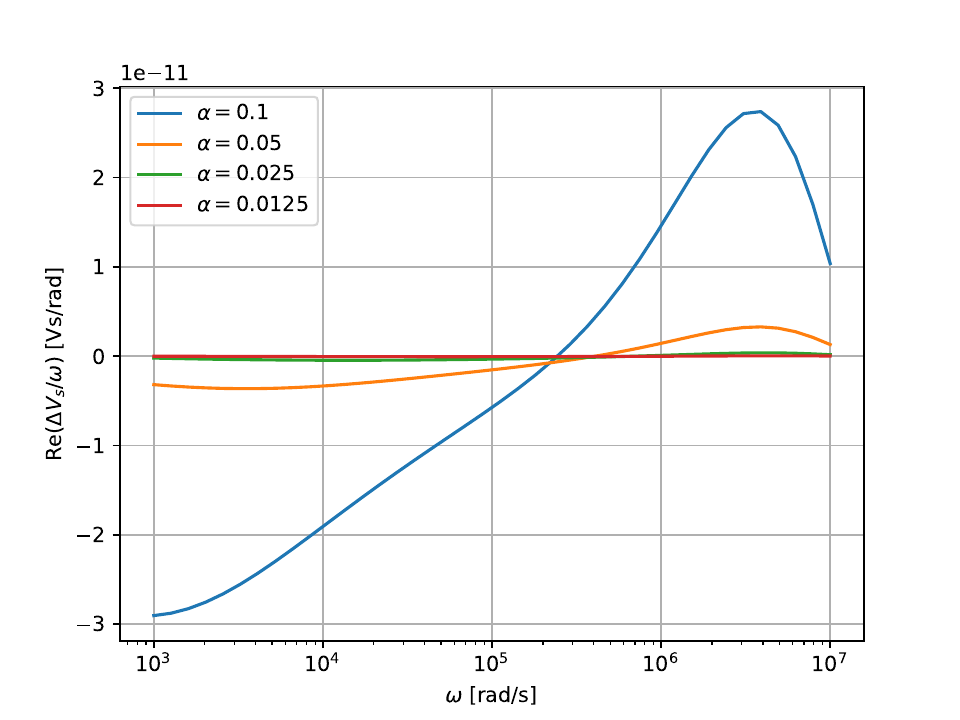} &
 \includegraphics[width=3in]{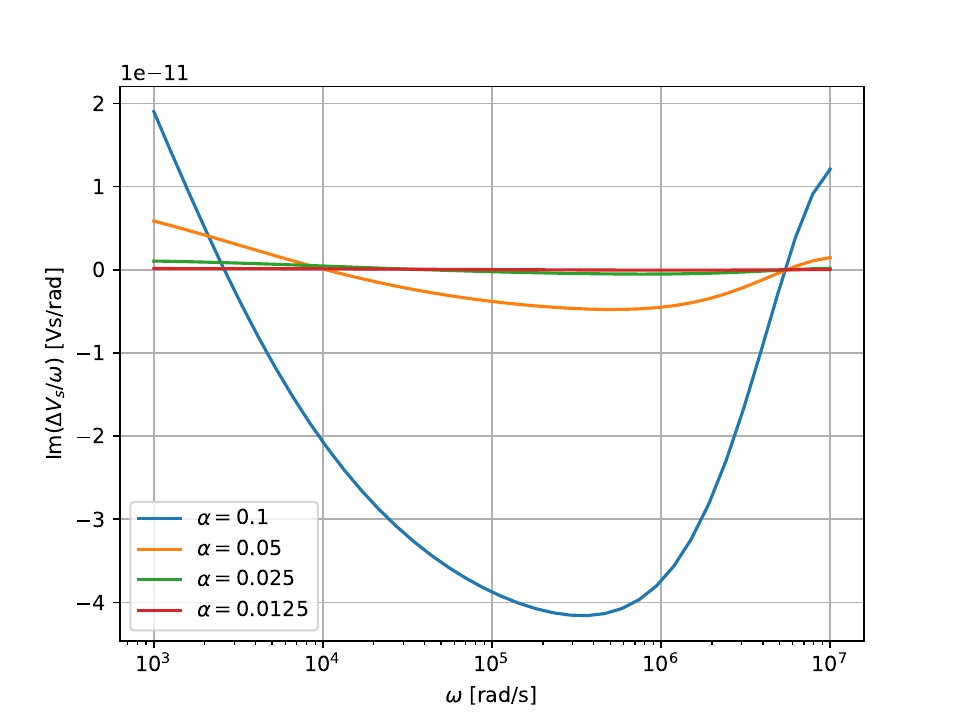} \\
 (a) & (b)
\end{array}$
 \end{center}
 \caption{ {Conducting buried sphere:  Influence of the object's size on $\Delta V_s$ in the measurement coil as a function of $\omega$. Showing $(a)$ $\text{Re}( \Delta V_s /\omega)$ and $(b)$ $\text{Im}( \Delta V_s /\omega)$.}} \label{fig:voltcomparealphasphere}
 \end{figure}

Approximate residual computations as a function of $1/\alpha$ are shown in Figure~\ref{fig:compareressphere} for two object locations ${\vec z}= (0,0,-0.3)$  and ${\vec z}= (0,0,-0.4)$. These illustrate that the rate exceeds or are the same as $C\alpha^4$ as $\alpha \to 0$, which is the expected behaviour.

 \begin{figure}
 \begin{center}
 $\begin{array}{cc}
 \includegraphics[width=3in]{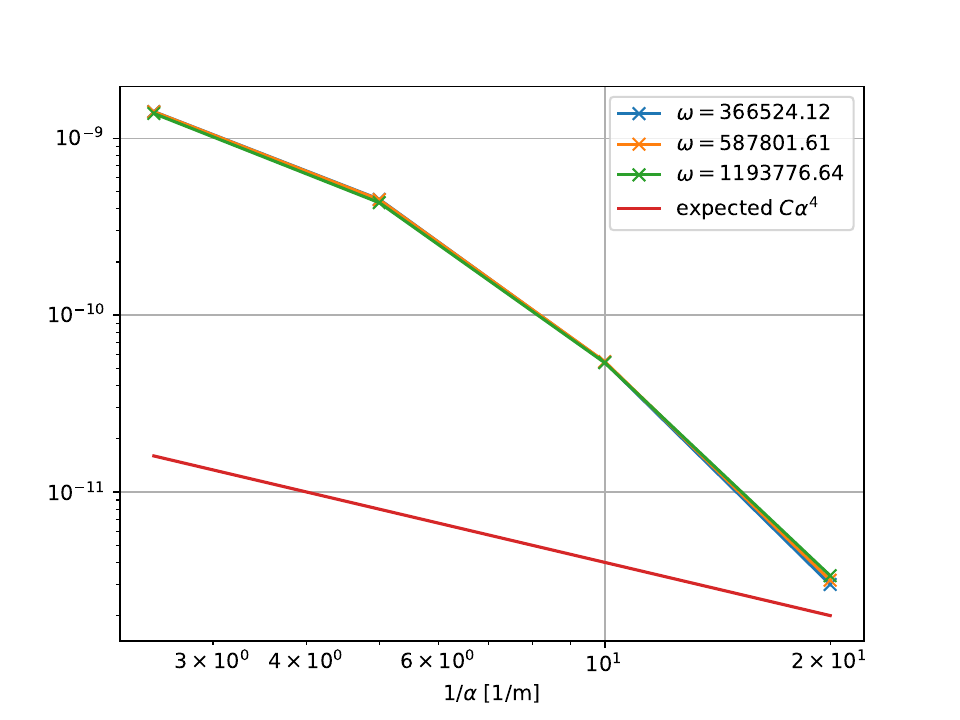} &
 \includegraphics[width=3in]{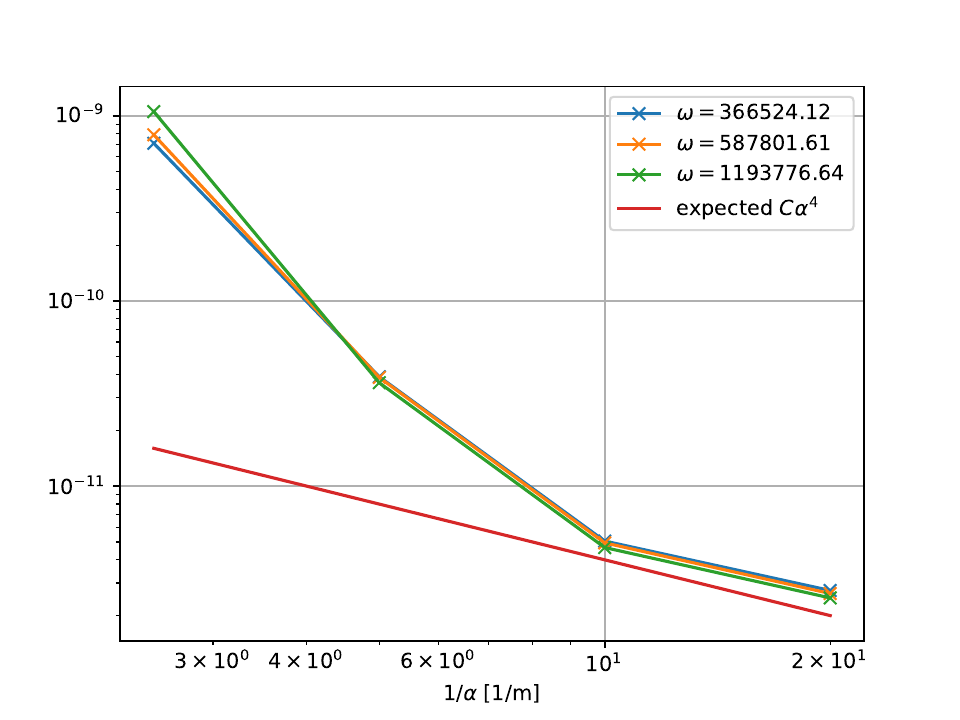} \\
  (a) & (b)
\end{array}$
 \end{center}
 \caption{ {Conducting buried sphere:  Approximate residual as a function of  $1/\alpha$ for two object locations. Showing $(a)$ ${\vec z}= (0,0,-0.3)$ and $(b)$ ${\vec z}= (0,0,-0.4)$.}} \label{fig:compareressphere}
 \end{figure}
 
\subsection{Buried ring object}

 {The situation is as described in Section~\ref{results:sphere} except that the buried spherical object is now replaced by a small ring. The ring is chosen to be of inner and outer radii 0.01 m and 0.012 m, respectively, with height $5\times 10^{-3}$ m and to be made of gold  with $\sigma_*=4.26\times 10^7$ S/m and $\mu_{r,*}=1$. The soil is chosen to have properties $\sigma_s=1. 6$ S/m and $\mu_{r,s}=1.0006$.  The discretisation is similar to that illustrated in Figure~\ref{fig:sim1mesh}, albeit with the sphere replaced by a smaller ring.
 The coefficients of the MPT object characterisation  $({\mathcal M})_{ij}$ are in this case obtained using the approach described in~\cite{Elgy2024_preprint}. As part of this process, approximate solutions to ${\vec \theta}_i$ to (\ref{eqn:thetatrans}) were computed at each $\omega$ of interest. A visualisation of the approximate solutions for $\omega = 1 \times 10^5$ rad/s in the vicinity of the object is shown in Figure~\ref{fig:ringcont}. The vectorial solutions as a function of $\omega$ are used to compute the MPT spectral signature object characterisation shown in Figure~\ref{fig:ringmpt}, which has two non-zero independent coefficients $({\mathcal M})_{11} = ({\mathcal M})_{22}$ and $({\mathcal M})_{33}  $, whose real and imaginary parts vary as a function of exciting frequency as shown.}

\begin{figure}
\begin{center}
$\begin{array}{ccc}
\includegraphics[width=1.4in]{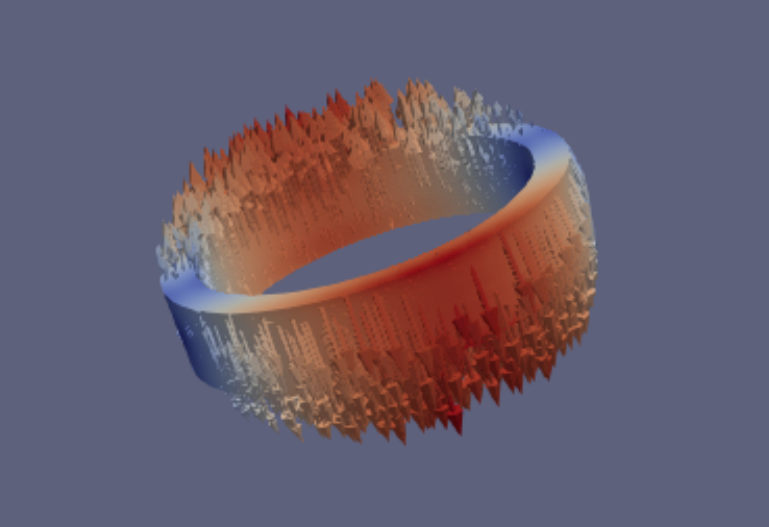} &
\includegraphics[width=1.5in]{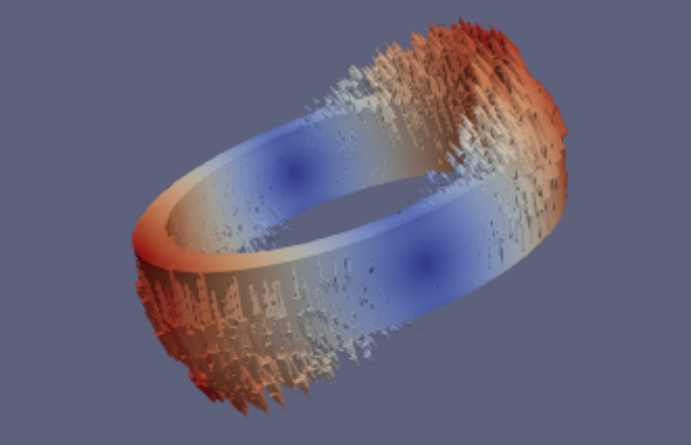} &
 \includegraphics[width=1.6in]{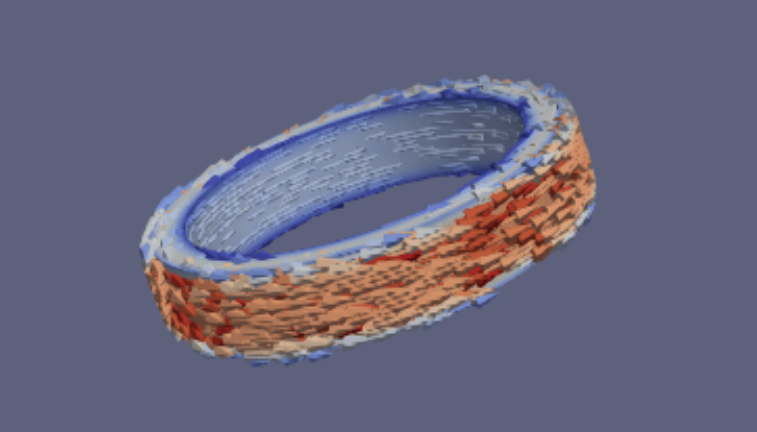} \\
 (a) & (b) &(c)
\end{array}$
\end{center}
\caption{ {Conducting buried ring: computed approximate solutions
$(a)$ $ \text{Im}({\vec \theta}_1^{(1)})$,$(b)$  $ \text{Im}({\vec \theta}_2^{(1)})$ and $(c)$ $ \text{Im}({\vec \theta}_3^{(1)})$
 to the transmission problem (\ref{eqn:thetatrans}) for object characterisation  using the splitting ${\vec \theta}_i = \tilde{\vec \theta}_i^{(0)} +{\vec \theta}_i^{(1)}$  for $\omega=1\times 10^5$ rad/s.}}\label{fig:ringcont}
\end{figure}
 {Using these MPT object characterisations, and considering  objects placed at ${\vec z}=(0,0,-0.2)$ m and ${\vec z}=(0,0,-0.4)$ m, the investigation shown Figure~\ref{fig:voltcomparemethsphere} is repeated for the case of the ring with results presented in Figure~\ref{fig:voltcomparemethring}.}
\begin{figure}
\begin{center}
 $\begin{array}{cc}
 \includegraphics[width=3in]{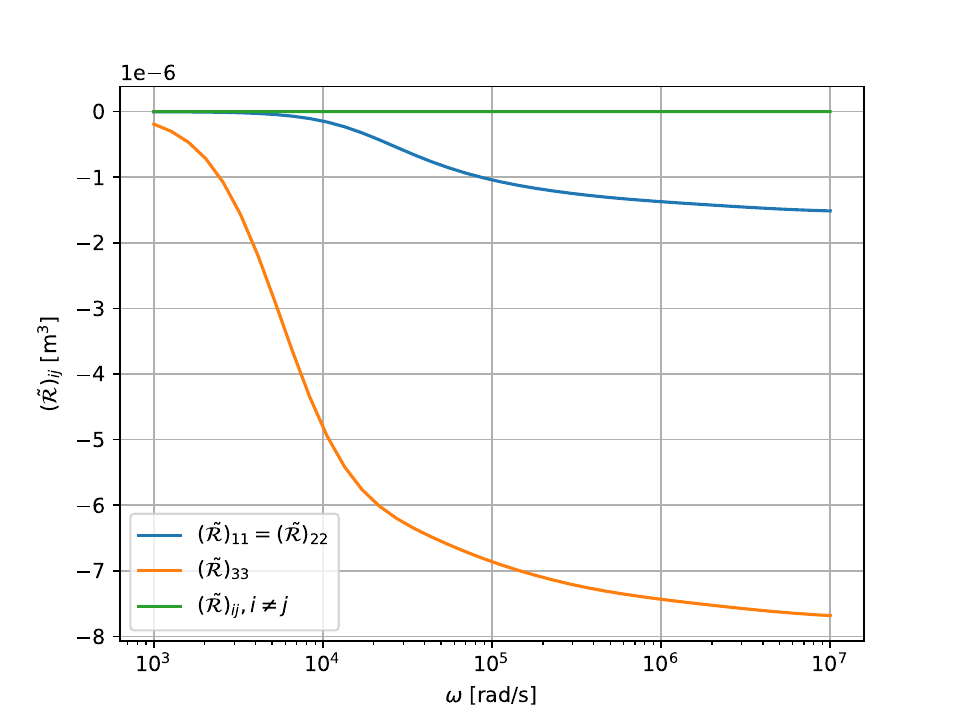} &
 \includegraphics[width=3in]{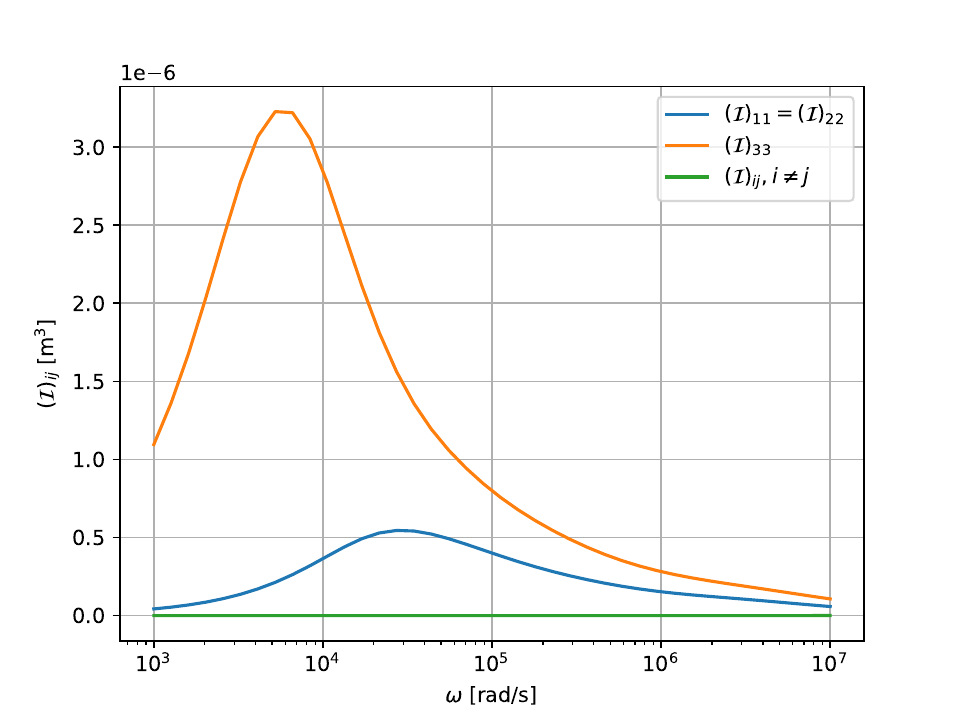} \\
  (a) & (b)
  \end{array}$
\end{center}
\caption{ {Conducting buried ring:  Computed MPT spectral signature $(a)$ $(\tilde{\mathcal R} )_{ij}$ and $(b)$ $({\mathcal I})_{ij}$.}}\label{fig:ringmpt}
\end{figure}
 { The comparisons are similar to the aforementioned spherical object, where the importance of including the soil being clearly visible, with similar levels of accuracy to the previous case. As expected, the voltages become weaker for the case where ${\vec z}=(0,0,-0.4)$ m.}

   \begin{figure}
 \begin{center}
 $\begin{array}{cc}
 \includegraphics[width=3in]{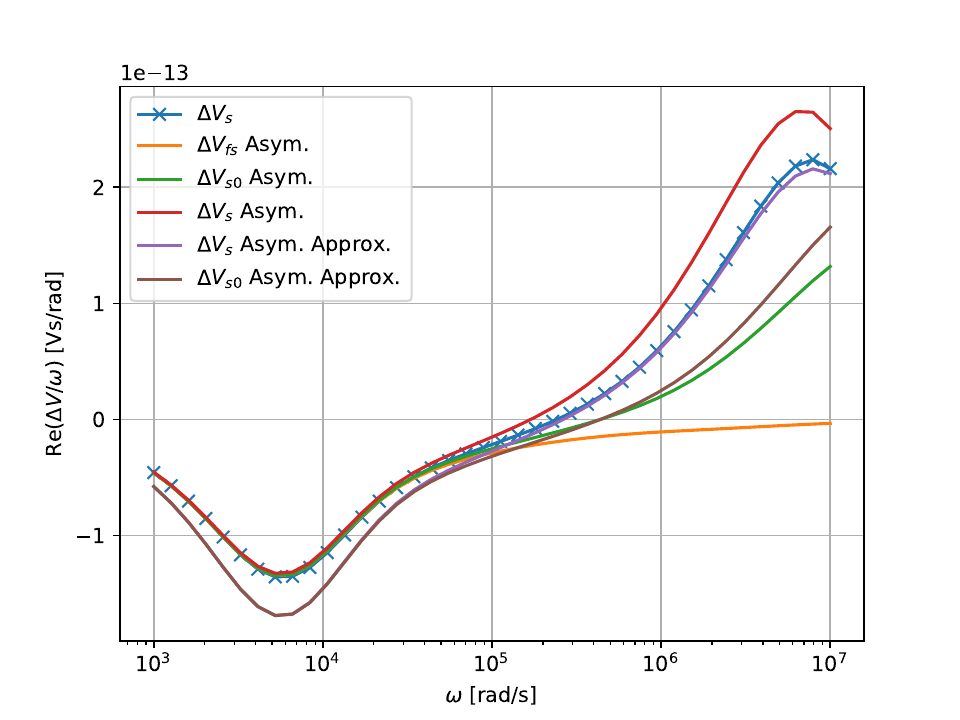} &
 \includegraphics[width=3in]{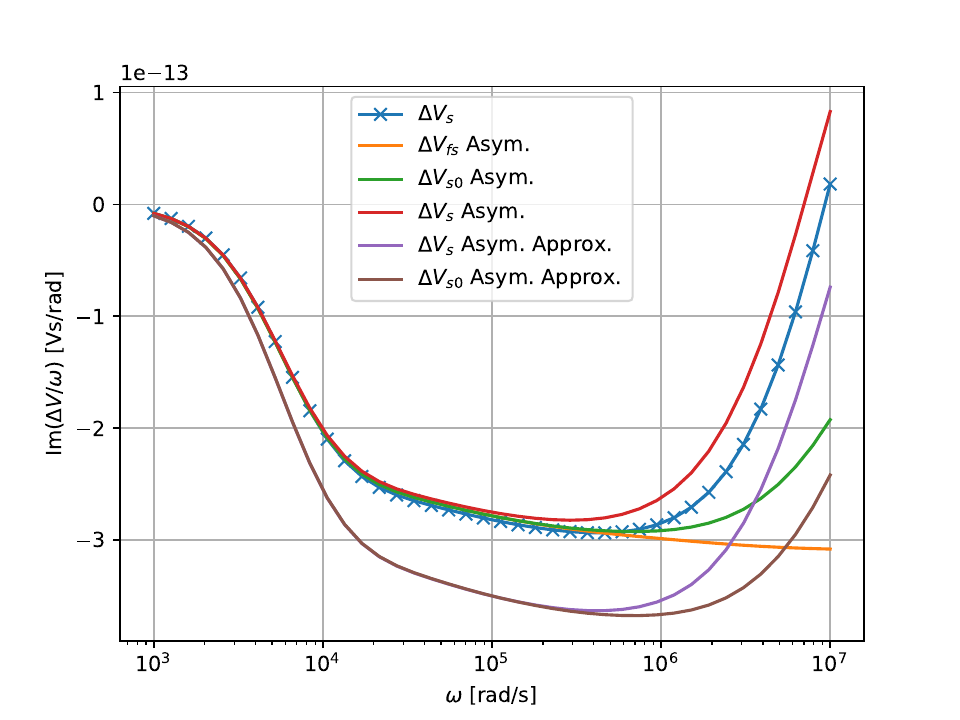}  \\
(a) & (b) \\
 \includegraphics[width=3in]{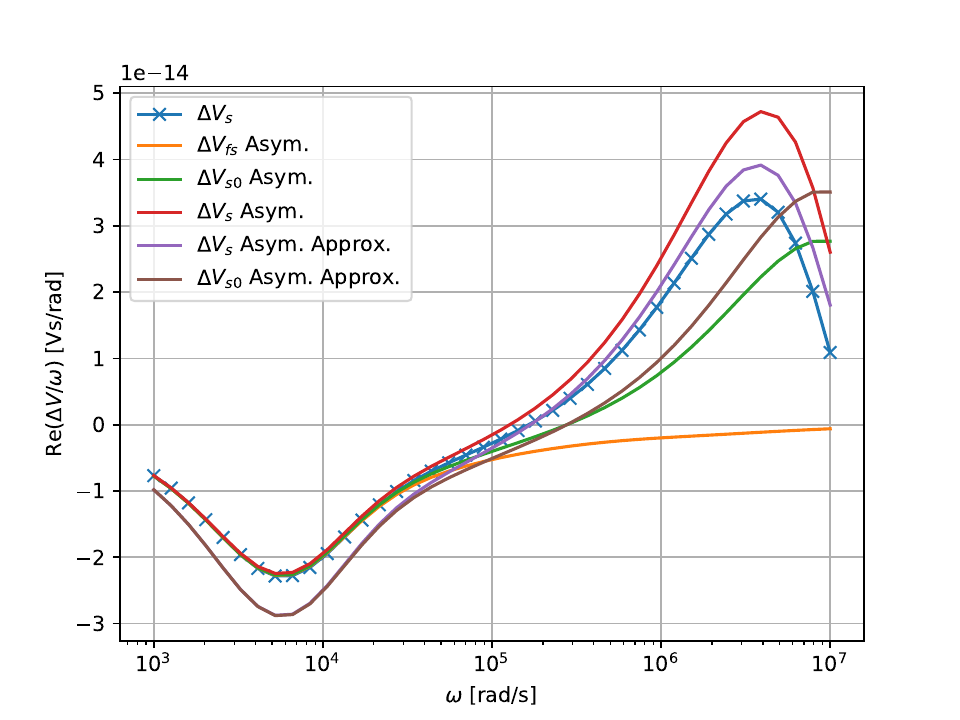} &
 \includegraphics[width=3in]{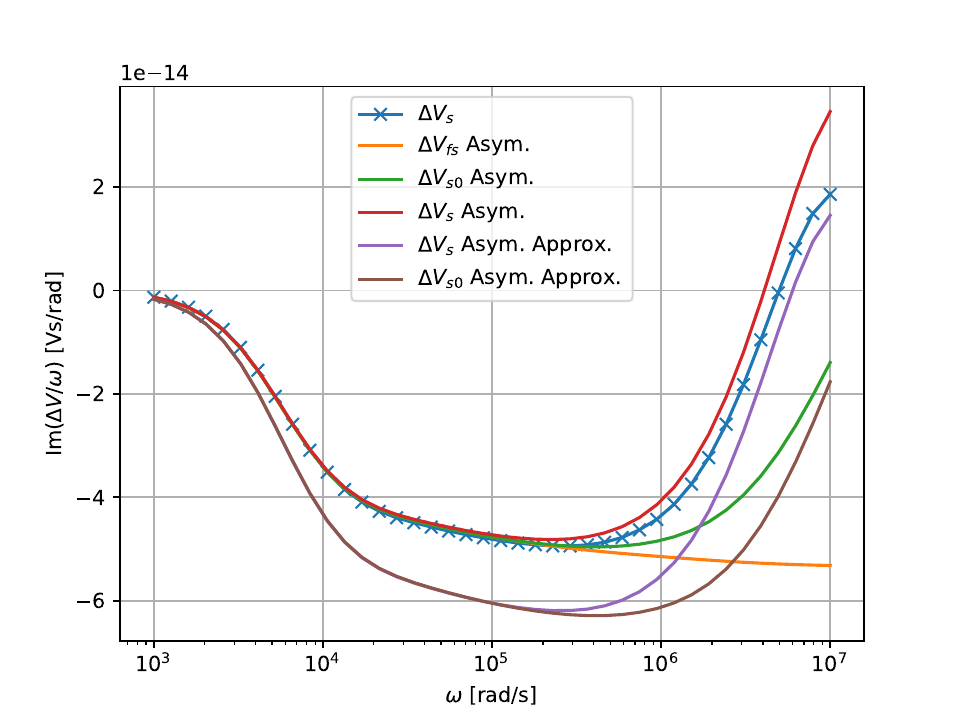}  \\
(c) & (d) 
 \end{array}$
 \end{center}
 \caption{ {Conducting buried ring:  Comparison of induced voltages $\Delta V_s$ and  $\Delta V_{fs}$ and  $\Delta V_{s0}$ in the measurement coil as a function of $\omega$. Showing $(a)$ $\text{Re}( \Delta V /\omega)$ and $(b)$ $\text{Im}( \Delta V /\omega)$ for ${\vec z}=(0,0,-0.2)$ m and  $(c)$ $\text{Re}( \Delta V /\omega)$ and $(d)$ $\text{Im}( \Delta V /\omega)$ for ${\vec z}=(0,0,-0.4)$ m. }} \label{fig:voltcomparemethring}
 \end{figure}

 {Approximate residual computations as a function of $1/\alpha$ for the buried ring are shown in Figure~\ref{fig:comparereflatring} for two object locations ${\vec z}= (0,0,-0.3)$  and ${\vec z}= (0,0,-0.4)$, which are the same locations considered in Figure~\ref{fig:compareressphere} for the spherical object. Again these illustrate that the rate exceeds  $C\alpha^4$ as $\alpha \to 0$, which is the expected behaviour.} 
 
 \begin{figure}
 \begin{center}
 $\begin{array}{cc}
 \includegraphics[width=3in]{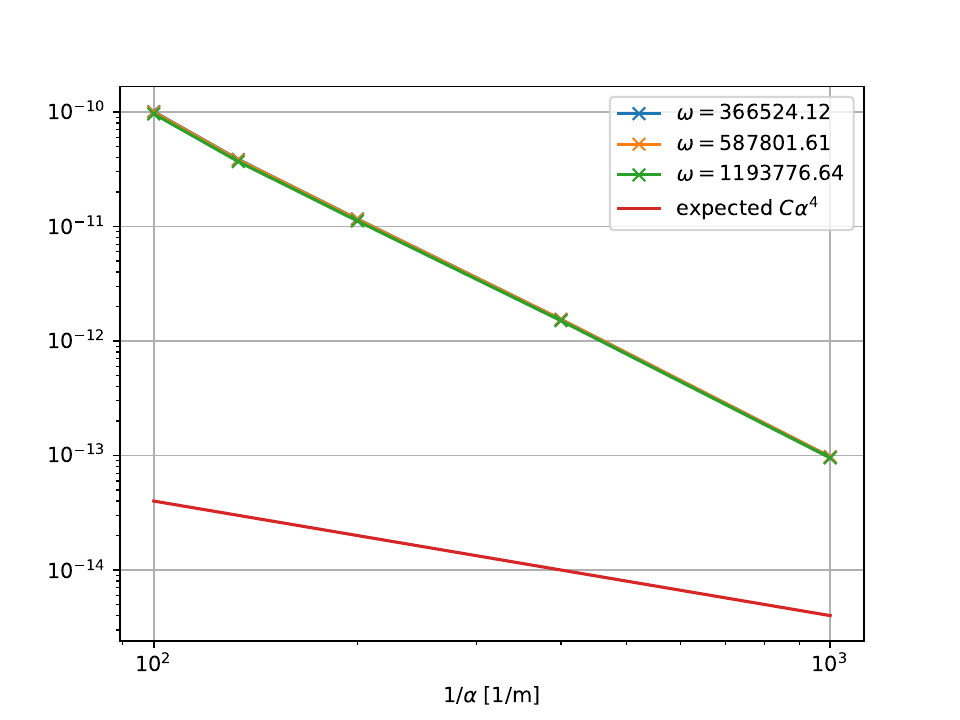} &
 \includegraphics[width=3in]{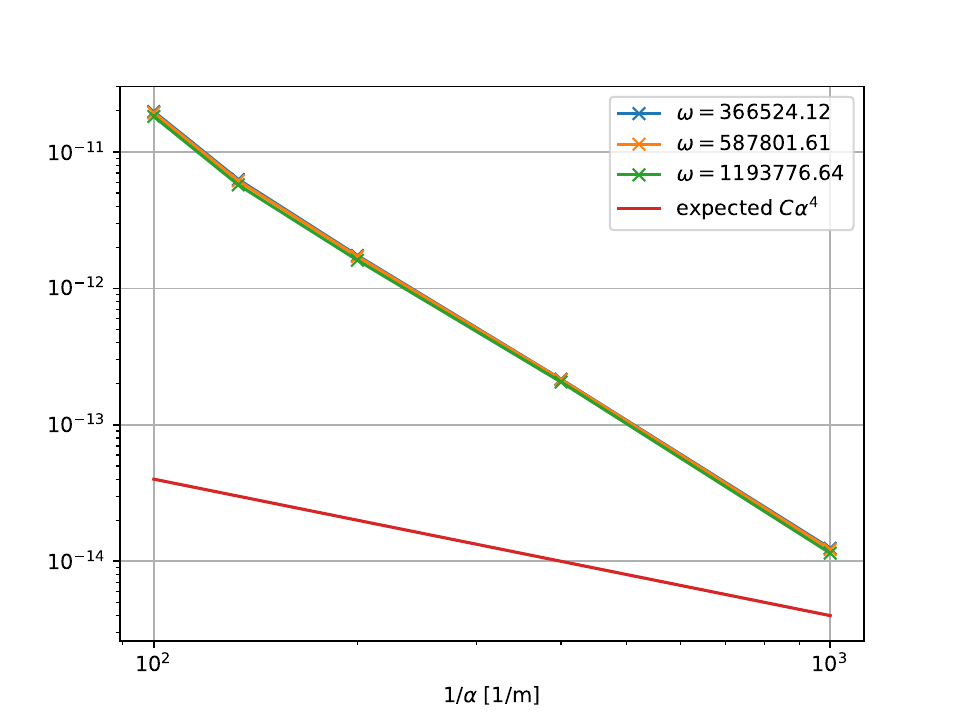} \\
  (a) & (b)
\end{array}$
 \end{center}
 \caption{ {Conducting buried ring:  Approximate residual as a function of  $1/\alpha$ for two object locations. Showing $(a)$ ${\vec z}= (0,0,-0.3)$ and $(b)$ ${\vec z}= (0,0,-0.4)$.}} \label{fig:comparereflatring}
 \end{figure}

\section{Conclusions}
 {In this paper we have derived an asymptotic expansion for the perturbed magnetic field in the presence of a
 small, arbitrarily shaped, highly conducting buried object with a smooth boundary and constant material parameters.}   {Previous work  assumed that the ground is non-conducting and with the same permeability as free space. We have removed  this key  restriction, which is important for buried objects like landmines and UXO,  and obtained new results on object characterisation.
We have provided conditions for the permeability and conductivity of the uncooperative}  {soil in order to allow the object to be characterised by a complex symmetric rank-2 MPT, whose coefficients can be computed independently of the object's position.
We have illustrated the importance of including the soil's material properties and the usefulness of our new result through a series of computational experiments.} {Future work will include the extension to characterisation of buried objects by generalised MPTs allowing for discrimination of more subtle classes of objects.}

\section*{Acknowledgements}

 {P.D. Ledger gratefully acknowledges the financial support received through an ICMS KE\_Catalyst Award, which has supported this work.}  {W.R.B. Lionheart gratefully acknowledges  Clare Hall, Cambridge for a visiting fellowship.}

\section*{Competing interests}
The authors declare none.

\bibliographystyle{acm}
\bibliography{paperbib}

\begin{thebibliography}{10}

\bibitem{ammaribuffa2000}
{\sc Ammari, H., Buffa, A., and {N\'ed\'elec}, J.-C.}
\newblock A justification of the eddy current model for the {M}axwell's
  equations.
\newblock {\em SIAM Journal on Applied Mathematics 60}, 5 (2000), 1805--1823.

\bibitem{Ammari2014}
{\sc Ammari, H., Chen, J., Chen, Z., Garnier, J., and Volkov, D.}
\newblock Target detection and characterization from electromagnetic induction
  data.
\newblock {\em Journal de Math{\'e}matiques Pures et Appliqu{\'e}es 101(1)\/}
  (2014), 54--75.

\bibitem{Ammari2015}
{\sc Ammari, H., Chen, J., Chen, Z., Volkov, D., and Wang, H.}
\newblock Detection and classification from electromagnetic induction data.
\newblock {\em Journal of Computational Physics 301\/} (2015), 201--217.

\bibitem{ammarikangbook}
{\sc Ammari, H., and Kang, H.}
\newblock {\em Polarization and Moment Tensors with Applications to Inverse
  Problems and Effective Medium Theory}.
\newblock Springer-Verlag New York, 2007.

\bibitem{bachlnger2025}
{\sc Bachinger, F., Langer, U., and Sch{\"o}berl, J.}
\newblock Numerical analysis of nonlinear multiharmonic eddy current problems.
\newblock {\em Numerische Mathematik 100\/} (2005), 593–616.

\bibitem{bowler1987}
{\sc Bowler, J.}
\newblock Eddy current calculations using halfspace {Green's} functions.
\newblock {\em Journal of Applied Physics 61\/} (1987), 833--839.

\bibitem{bowler2004}
{\sc Bowler, J.}
\newblock Study notes in electromagnetic {NDE}.
\newblock Tech. rep., 2004.

\bibitem{bowler1998}
{\sc Bowler, J., and Harfield, N.}
\newblock Evaluation of probe impedance due to thin eddy current interaction
  with surface cracks.
\newblock {\em IEEE Transactions on Magnetics 34\/} (1998), 515--523.

\bibitem{bowler1991}
{\sc Bowler, J., Jenkins, S., Sabbagh, L., and Sabbagh, H.}
\newblock Eddy-current probe impedance due to a volumetric flaw.
\newblock {\em Journal of Applied Physics 70}, 3 (1991).

\bibitem{Elgy2024_preprint}
{\sc Elgy, J., and Ledger, P.~D.}
\newblock Efficient computation of magnetic polarizability tensor spectral
  signatures for object characterisation in metal detection.
\newblock {\em Engineering Computations 41\/} (2024), 2472--2503.

\bibitem{LedgerLionheart2015}
{\sc Ledger, P.~D., and Lionheart, W.~R.~B.}
\newblock Characterising the shape and material properties of hidden targets
  from magnetic induction data.
\newblock {\em IMA Journal of Applied Mathematics 80(6)\/} (2015), 1776--1798.

\bibitem{LedgerLionheart2018}
{\sc Ledger, P.~D., and Lionheart, W.~R.~B.}
\newblock An explicit formula for the magnetic polarizability tensor for object
  characterization.
\newblock {\em {IEEE} Transactions on Geoscience and Remote Sensing 56(6)\/}
  (2018), 3520--3533.

\bibitem{LedgerLionheart2018g}
{\sc Ledger, P.~D., and Lionheart, W.~R.~B.}
\newblock Generalised magnetic polarizability tensors.
\newblock {\em Mathematical Methods in the Applied Sciences 41\/} (2018),
  3175--3196.

\bibitem{LedgerLionheart2020spect}
{\sc Ledger, P.~D., and Lionheart, W.~R.~B.}
\newblock The spectral properties of the magnetic polarizability tensor for
  metallic object characterisation.
\newblock {\em Mathematical Methods in the Applied Sciences 43\/} (2020),
  78--113.

\bibitem{LedgerLionheart2023a}
{\sc Ledger, P.~D., and Lionheart, W.~R.~B.}
\newblock Properties of generalized magnetic polarizability tensors.
\newblock {\em Mathematical Methods in the Applied Sciences 43\/} (2023),
  5604--5631.

\bibitem{LedgerLionheart2024}
{\sc Ledger, P.~D., and Lionheart, W.~R.~B.}
\newblock Characterising small objects in the regime between the eddy current
  model and wave propagation.
\newblock {\em European Journal of Applied Mathematics 35\/} (2024), 294--317.

\bibitem{ledgerwilsonamadlion2021}
{\sc Ledger, P.~D., Wilson, B.~A., Amad, A.~A.~S., and Lionheart, W.~R.~B.}
\newblock Identification of metallic objects using spectral magnetic
  polarizability tensor signatures: Object characterisation and invariants.
\newblock {\em International Journal for Numerical Methods in Engineering
  122\/} (2021), 3941--3984.

\bibitem{ledgerwilsonlion2022}
{\sc Ledger, P.~D., Wilson, B.~A., and Lionheart, W.~R.~B.}
\newblock Identification of metallic objects using spectral magnetic
  polarizability tensor signatures: Object classification.
\newblock {\em International Journal for Numerical Methods in Engineering
  123\/} (2022), 2076--2111.

\bibitem{ledgerzaglmayr2010}
{\sc Ledger, P.~D., and Zaglmayr, S.}
\newblock {$hp$}-finite element simulation of three-dimensional eddy current
  problems on multiply connected domains.
\newblock {\em Computer Methods in Applied Mechanics and Engineering 199\/}
  (2010), 3386--3401.

\bibitem{marsh2014b}
{\sc Makkonen, J., Marsh, L.~A., Vihonen, J., {J\"arvi}, A., Armitage, D.~W.,
  Visa, A., and Peyton, A.~J.}
\newblock {KNN} classification of metallic targets using the magnetic
  polarizability tensor.
\newblock {\em Measurement Science and Technology 25\/} (2014), 055105.

\bibitem{marsh2015}
{\sc Makkonen, J., Marsh, L.~A., Vihonen, J., {J\"arvi}, A., Armitage, D.~W.,
  Visa, A., and Peyton, A.~J.}
\newblock Improving reliability for classification of metallic objects using a
  {WTMD} portal.
\newblock {\em Measurement Science and Technology 26\/} (2015), 105103.

\bibitem{raiche1974}
{\sc Raiche, A.}
\newblock An integral equation approach to three-dimensional modelling.
\newblock {\em Geophysical Journal of the Royal Astronomical Society 36\/}
  (1974), 363--376.

\bibitem{raiche1975}
{\sc Raiche, A., and Coggon, J.}
\newblock Analytic {G}reen's tensors for integral equation modelling.
\newblock {\em Geophysical Journal of the Royal Astronomical Society 42\/}
  (1975), 1035--1038.

\bibitem{netgendet}
{\sc {Sch\"oberl}, J.}
\newblock {NETGEN} - an advancing front {2D/3D}-mesh generator based on
  abstract rules.
\newblock {\em Computing and Visualization in Science 1(1)\/} (1997), 41--52.

\bibitem{SchoberlZaglmayr2005}
{\sc {Sch\"{o}berl}, J., and Zaglmayr, S.}
\newblock High order {N\'{e}d\'{e}lec} elements with local complete sequence
  properties.
\newblock {\em COMPEL-The International Journal for Computation and Mathematics
  in Electrical and Electronic Engineering 24(2)\/} (2005), 374--384.

\bibitem{magsoilperm}
{\sc Scott, J.}
\newblock Electrical and magnetic properties of rock and soil.
\newblock \url{https://pubs.usgs.gov/of/1983/0915/report.pdf}, 1983.
\newblock Report 83-915, Accessed 4/6/25.

\bibitem{salinesoilcond}
{\sc Service, U. N. R.~C.}
\newblock Soil quality indicators.
\newblock
  \url{https://www.nrcs.usda.gov/sites/default/files/2022-10/Soil%20Electrical%20Conductivity.pdf},
  2011.
\newblock Accessed 4/6/25.

\bibitem{Wait1951}
{\sc Wait, J.~R.}
\newblock A conducting sphere in a time varying magnetic field.
\newblock {\em Geophysics 16(4)\/} (1951), 666--672.

\bibitem{ben2020}
{\sc Wilson, B.~A., and Ledger, P.~D.}
\newblock Efficient computation of the magnetic polarizability tensor spectral
  signature using {POD}.
\newblock {\em International Journal for Numerical Methods in Engineering
  122\/} (2021), 1940--1963.

\bibitem{zaglmayrphd}
{\sc Zaglmayr, S.}
\newblock {\em High Order Finite Elements for Electromagnetic Field
  Computation}.
\newblock PhD thesis, Johannes Kepler University Linz, 2006.

\end{thebibliography}
\end{document}